\documentclass{amsart}
\usepackage{amsmath, amssymb}
\usepackage[all]{xypic}

\newtheorem{theorem}[equation]{Theorem}
\newtheorem{proposition}[equation]{Proposition}
\newtheorem{corollary}[equation]{Corollary}
\newtheorem{lemma}[equation]{Lemma}

\numberwithin{equation}{section}

\newenvironment{example}
	{\refstepcounter{equation}\medskip\noindent{\bf Example \theequation.}}
	{\medskip}
\newenvironment{remark}
	{\refstepcounter{equation}\medskip\noindent{\bf Remark \theequation.}}
	{\medskip}
\newenvironment{defn}
	{\refstepcounter{equation}\medskip\noindent{\bf Definition \theequation.}}
	{\medskip}

\newcommand{\beq}{\begin{equation}}
\newcommand{\eeq}{\end{equation}}

\DeclareMathOperator{\id}{Id}

\DeclareMathOperator{\Hom}{Hom}
\renewcommand{\hom}{{\rm hom}}

\newcommand{\End}{{\rm End}}

\newcommand{\Ext}{{\rm Ext}}
\newcommand{\ext}{{\rm ext}}

\DeclareMathOperator{\Aut}{Aut}
\DeclareMathOperator{\Inn}{Inn}

\newcommand{\zed}{{\mathbb Z}}

\newcommand{\shift}{\mathcal{S}}

\newcommand{\ang}[1]{\langle #1 \rangle}

\newcommand{\blank}{\mbox{$\underline{\makebox[10pt]{}}$}}

\DeclareMathOperator{\rgr}{gr-\!}
\DeclareMathOperator{\lgr}{\!-gr}
\DeclareMathOperator{\rmod}{mod-\!}
\DeclareMathOperator{\lmod}{\!-mod}
\DeclareMathOperator{\rmodu}{mod^{u}-\!}

\newcommand{\st}{\left\vert\right.}

\newcommand{\I}{{\mathbb I}}

\DeclareMathOperator{\coker}{Coker}
\DeclareMathOperator{\im}{Im}
\newcommand{\bbar}[1]{\overline{#1}}

\newcommand{\Hfunct}[3]{H_{\bbar{#1}}({#2},{#3})}
\newcommand{\Htwist}[4]{\Hom^{#1}_{#2}({#3}, {#4})}

\newcommand{\inv}{\iota}
\newcommand{\invbrak}[1]{\inv[#1]}
\newcommand{\zedfin}{\zed_{{\rm fin}}}

\newcommand{\weak}{\cong_w}

\newcommand{\sh}{\mathcal}

\newcommand{\F}{{\mathcal F}}

\DeclareMathOperator{\GWA}{W}

\newcommand{\ver}[1]{^{(#1)}}

\DeclareMathOperator{\Pic}{Pic}

\DeclareMathOperator{\Picz}{\Pic_0}
\DeclareMathOperator{\Supp}{Supp}
\DeclareMathOperator{\Spec}{Spec}
\DeclareMathOperator{\Proj}{Proj}

\DeclareMathOperator{\mywr}{wr}

\newcommand{\X}[1][0]{X \ang{#1}}
\newcommand{\Y}[1][0]{Y \ang{#1}}

\newcommand{\brackarr}[1]
{ {\renewcommand{\arraystretch}{1.5} \left\{   \begin{array}{ll} #1 \end{array}  \right.}}

\title{Rings graded equivalent to the Weyl algebra}
\author{Susan J. Sierra}
\thanks{Department of Mathematics, 
University of Washington,
Seattle, WA 98195-4350.  {\em Email: } {\tt sjsierra@math.washington.edu}}
\date{\today}

\begin{document}
\begin{abstract}
We consider the first Weyl algebra, $A$, in the Euler gradation, and completely classify graded rings $B$ that are graded equivalent to $A$: that is, the categories $\rgr A$ and $\rgr B$ are equivalent.  This includes some surprising examples:  in particular, we show that $A$ is graded equivalent to an idealizer in a localization of $A$.

We obtain this classification as an application of a general Morita-type characterization of equivalences of graded module categories.   
  Given a $\zed$-graded ring $R$, an autoequivalence $\F$ of $\rgr R$, and a finitely generated graded projective right $R$-module $P$, we show how to construct a {\em twisted endomorphism ring} $\End^\F_R(P)$ and prove:

\medskip
\noindent{\bf Theorem.}
{\it The $\zed$-graded rings $R$ and $S$ are graded equivalent if and only if there are an autoequivalence $\F$ of $\rgr R$ and a finitely generated graded projective right $R$-module $P$ such that the modules $\{ \F^n P \}$ generate $\rgr R$ and  $S \cong \End^\F_R(P)$.}
\medskip
\end{abstract}

\maketitle
\tableofcontents
\section{Introduction}\label{sec-intro}
The first Weyl algebra $A = k \{ x,y \} /(xy-yx-1)$, where $k$ is an algebraically closed  field of characteristic 0, is in many ways the fundamental ring of noncommutative algebra.  Remarkably, although it has been studied since the 1930's, it continues to inspire new developments in the field.

One of the most fruitful areas of recent activity has come from considering rings Morita equivalent to $A$.  
The starting point of these results  was a theorem of Stafford \cite[Corollary~B]{St2}:  up to isomorphism, the domains Morita equivalent to $A$ are in natural bijection with the orbits of $\Aut_k(A)$ on the set of (module) isomorphism classes of right ideals of $A$.     Using earlier work of Cannings and Holland \cite{CH}, Berest and Wilson  \cite{BW1} then exhibited a rich geometry related to Stafford's result.  They
showed \cite[Theorem~1.1]{BW1} that isomorphism classes of right ideals of $A$ correspond to points in the {\em Calogero-Moser space} $\mathcal{C} = \bigcup_{N \geq 0} \mathcal{C}_N$, where each $\mathcal{C}_N$ should be thought of as a noncommutative deformation of the the Hilbert scheme of $N$ points in the affine 2-plane, and that $\Aut_k(A)$ also acts naturally on $\sh{C}$.   Further, they proved \cite[Theorem~1.2]{BW1} that the orbits of $\Aut_k(A)$ on $\mathcal{C}$ are precisely the $\mathcal{C}_N$.  (The fact that the orbits of $\Aut_k(A)$ acting on the set of isomorphism classes of right ideals are parameterized by the non-negative integers was independently proved by Kouakou \cite{K}.)    Thus we have:

\begin{theorem}[Berest-Wilson, Kouakou, Stafford]\label{thm-BKSW}
The isomorphism classes of domains Morita equivalent to $A$ are in 1-1 correspondence with the non-negative integers. \qed
\end{theorem}

In this paper, we solve the analogous problem for {\em graded} module categories.  That is, we grade $A$ by giving $x$ degree 1 and $y$ degree -1.   We ask:  what are the $\zed$-graded rings $B$ such that the category $\rgr A$ of  graded right $A$-modules is equivalent to the category $\rgr B$ of graded right $B$-modules?
 We will call such an equivalence a {\em graded equivalence} and  will say that $A$ and $B$ are {\em graded equivalent}, although we caution that this terminology is not universally agreed upon in the literature.

Of course, any progenerator $P$ for the full module category of $A$ that happens to be graded induces a graded equivalence between $A$ and $\End_A(P)$; we refer to such an equivalence of categories as a {\em graded Morita equivalence}.  
Since we are interested in non-trivial graded equivalences, we ask:  
 what are the graded Morita equivalence classes of $\zed$-graded rings $B$ such that $A$ and $B$ are graded equivalent?

The answer to this question is quite surprising.   Given a positive integer $n$ and a set $J \subseteq \{0, \ldots, n-1\}$, define a ring $S(J,n)$ as follows:  first, let
\[ f = \prod^{n-1}_{\substack{i=0 \\ i \not\in J}} (z+i).\]
Let $W=W(J)$ be the ring generated over $k$ by $X$, $Y$, and $z$, with relations

\begin{align*}
Xz - zX & = n X	& Yz - zY & = -nY\\
XY & = f		& YX &  = f(z-n).
\end{align*}

Define 
\[S(J, n) = k[z] + \Bigl( \prod_{i \in J} (z+i) \Bigr) W \subseteq W.\]
Then we prove:
\begin{theorem}\label{ithm1}
{\em(Theorem~\ref{thm-class})}
Let $S$ be a $\zed$-graded ring.  Then $S$ is graded equivalent to $A$ if and only if $S$ is graded Morita equivalent to some $S(J, n)$.  
\end{theorem}

Furthermore, we give  a precise enumeration of the graded Morita equivalence classes of rings  graded equivalent to $A$:  
\begin{theorem}\label{ithm-neck}
{\em (Corollary~\ref{cor-necklacedata})}
The graded Morita equivalence classes of rings graded equivalent to $A$ are in 1-1 correspondence with pairs $(\mathbb{J}, n)$ where $n$ is a positive integer and $\mathbb{J}$ is a necklace of $n$ black and white beads.
\end{theorem}

 Theorem~\ref{ithm1} says in particular that $A \cong S(\emptyset, 1)$ is graded equivalent to $B = k + x A[y^{-1}] \cong S(\{0\},1)$.  The ring $B$ is the {\em idealizer} in $A[y^{-1}]$ of the right ideal $x A[y^{-1}]$; that is, $B = \{ \theta \in A[y^{-1}] \st \theta x \in x A[y^{-1}] \}$ is the maximal subring of $A[y^{-1}]$ in which $x A[y^{-1}]$ is a two-sided ideal.    While it is known that simplicity is not necessarily preserved under graded equivalence, it is still unexpected to find an equivalence between $A$ and an idealizer in a localization of $A$.  For example, $B$ has a 1-dimensional graded representation, while all $A$-modules are infinite-dimensional. By construction, $B$  has a nontrivial two-sided ideal, and is not a {\em maximal order}, while $A$, as is well-known, is simple.  In fact, $B$ is a standard example of a ring that fails the {\em second layer condition} governing relationships between prime ideals, whereas $A$ trivially satisfies the second layer condition.    (See the discussion of Example~\ref{eg-ring0} for definitions.)  Thus we obtain:

\begin{corollary}\label{icor2}
The properties of having a finite dimensional graded representation, being a maximal order, and satisfying the second layer condition are not invariant under graded equivalence. \qed
\end{corollary}

Another ring occurring in Theorem~\ref{ithm1} is the {\em Veronese ring} $A\ver{2} = \bigoplus_{n \in \zed} A_{2n} \cong S(\emptyset, 2)$.   By Theorem~\ref{ithm1}, $A$ and $A\ver{2}$ are graded equivalent.  Of course one expects that $\Proj A$ (in the appropriate sense) and $\Proj A\ver{2}$ will be equivalent,  but, as far as we  are aware, this is the first nontrivial example of an equivalence  between the graded module categories of  a ring and its Veronese.

Theorem~\ref{ithm1} is an  application of general results on equivalences of graded module categories.  Given a graded ring $R$, an autoequivalence $\F$ of $\rgr R$, and a finitely generated graded right $R$-module $P$, we show that there is a natural way to construct a {\em twisted endomorphism ring $\End^\F_R(P)$}.   We prove:
\begin{theorem}\label{ithm-grMor}
{\em(Theorem~\ref{thm-grMor-pullback})}
Let $R$ and $S$ be $\zed$-graded rings.  Then $R$ and $S$ are graded equivalent if and only if there are a finitely generated graded projective right $R$-module $P$ and an autoequivalence $\F$ of $\rgr R$ such that $\{ \F^nP\}_{n \in \zed}$ generates $\rgr R$ and $S \cong \End^{\F}_R(P)$.
\end{theorem}
If $\F$ is the shift functor in $\rgr R$, then $\End_R^{\F}(P)$ is the classical endomorphism ring of $P$; thus Theorem~\ref{ithm-grMor} generalizes the Morita theorems.

If $R$ is a $\zed$-graded ring, we define (following \cite{BR}) the {\em Picard group} of $\rgr R$ to be the group of autoequivalences of $\rgr R$, modulo natural isomorphism.   
Theorem~\ref{ithm-grMor} illustrates the fundamental technique for all of our results: relate equivalences between graded module categories to the  Picard groups of the categories.    
Given $\zed$-graded rings $R$ and $S$, an equivalence of categories $\Phi: \rgr R \to \rgr S$ induces an isomorphism of Picard groups, which we write $\Phi^*: \Pic(\rgr S) \to \Pic(\rgr R)$; we refer to applying $\Phi^*$ as {\em pulling back along $\Phi$.}   Let $\F_\Phi$ be the autoequivalence of $\rgr R$ that comes from pulling back the shift functor in $\rgr S$ along $\Phi$.  It turns out that properties of $\Phi$ and $S$ can be deduced from the properties of $\F_\Phi$; in particular, $\F_\Phi$ is the autoequivalence $\F$ from Theorem~\ref{ithm-grMor}.

This paper is a companion to the work in \cite{S}.  
There, using results of \'Ahn-M\'arki \cite{AM}, we gave a Morita-type characterization of graded equivalences in terms of certain bigraded modules.
This characterization tells us that the rings graded equivalent to $A$ are endomorphism rings (in an appropriate sense) of these modules.  However, the modules are quite large and correspondingly difficult to work with.  One motivation for the current paper is  to develop techniques to make the analysis  in \cite{S} concrete in a specific case.

This paper has seven sections.  In Sections~\ref{sec-general} and \ref{sec-genPic} we prove Theorem~\ref{ithm-grMor} and other results that hold for any $\zed$-graded ring.  In particular, we characterize graded Morita equivalences and Zhang twists in terms of the Picard group.   In Sections~\ref{sec-category} and \ref{sec-Pic} we analyze the graded module category of the Weyl algebra and its Picard group.  The key result in these two sections is Corollary~\ref{cor-Picr}, describing the Picard group of $\rgr A$ explicitly. 

In Section~\ref{sec-main} we use the results of the previous two sections to prove Theorem~\ref{ithm1} and Theorem~\ref{ithm-neck}, and give several examples.  In  Section~\ref{sec-Ktheory} we describe the graded $K$-theory of $A$, and in particular show that, in contrast to the ungraded case, if $P \oplus Q \cong P \oplus Q'$ where $P$, $Q$, and $Q'$ are finitely generated graded projective modules, then $Q \cong Q'$.

In the remainder of the Introduction, we establish notation.    We fix a commutative ring $k$.  (Beginning in Section~\ref{sec-category}, $k$ will be an algebraically closed field of characteristic 0.)  For us, a {\em graded ring} means a $\zed$-graded $k$-algebra, and all categories and category equivalences are assumed to be $k$-linear.  If $R$ is a graded ring, the category $\rgr R$ consists of all $\zed$-graded right $R$-modules, and $\rmod R$ is the category of all right $R$-modules; similarly, we form the left module categories $R \lgr$ and $R \lmod$.  Morphisms in $\rgr R$ and $R \lgr$ are homomorphisms that fix degree; if $M$ and $N$ are graded $R$-modules, we write $\hom_R(M,N)$ to mean $\Hom_{\rgr R}(M, N)$, and write $\Hom_R(M, N)$ to denote  $\Hom_{\rmod R}(M,N)$.  We similarly write  $\Ext^i_R$ and $\ext^i_R$ for the derived functors of $\Hom_R$ and $\hom_R$, respectively.

If $R$ is  a graded ring, then the {\em shift functor} on $\rgr R$ (or $R \lgr$) sends a graded right or left module $M$ to the new module $M \ang{1} = \bigoplus_{j \in \zed} M\ang{1}_j$, defined by $M\ang{1}_j = M_{j-1}$.  We will write this functor  as:
\begin{align*}
\shift_R: 	M & \mapsto M\ang{1}.
	\end{align*}   
(We caution that this is the opposite of the standard convention.)
	
Since we will work  primarily with right module categories, for the remainder of the paper unless otherwise specified an $R$-module $M$ is a right module.  

 The research for this  paper was completed as  part of the author's Ph.D. studies at the University of Michigan under the direction of J.T. Stafford.  
The author was partially supported by NSF grants DMS-0502170 and  DMS-0802935.   In addition, the author would like to thank Paul Smith for many stimulating discussions, and the referee for careful reading, including pointing out an error in an earlier version, and several helpful suggestions.

\section{Graded equivalences and twisted endomorphism rings}\label{sec-general}

Let $R$ and $S$ be graded equivalent graded rings; we write $\rgr R \simeq \rgr S$.  Let $\Phi: \rgr R \to \rgr S$  be an equivalence of categories.  In this section, we show that, similarly to the classical Morita theorems, we may construct $S$ and $\Phi$ as  a ``twisted endomorphism ring'' and ``twisted Hom functor,'' respectively. 

 Let   $\Psi: \rgr S \to \rgr R$ be a quasi-inverse for $\Phi$ and let $\sh{G}$ be an autoequivalence of $\rgr S$.   We define  the {\em pullback of $\sh{G}$ along $\Phi$} to be the autoequivalence
 \[ \Phi^* \sh{G} = \Psi \sh{G} \Phi\] 
of $\rgr R$  ; likewise, we may {\em push forward} an autoequivalence $\F$ of $\rgr R$ to the autoequivalance $\Phi_* \F = \Phi \F \Psi$ of $\rgr S$.   

We will show that the ring $S$ and the equivalence of categories $\Phi$ may be reconstructed from $\Phi^* \shift_S$.  
To do this, let $P = \Psi(S)$.  Then we show that there is a natural way to define a {\em twisted Hom functor} 
$\Htwist{\F}{R}{P}{\blank}$ and a {\em twisted endomorphism ring} $\End^\F_R(P) = \Htwist{\F}{R}{P}{P}$ such that $S \cong \End^{\F}_R(P)$ and $\Phi$ is naturally isomorphic to $\Htwist{\F}{R}{P}{\blank}$.   We prove Theorem~\ref{ithm-grMor}, which, by characterizing graded equivalences in terms of twisted Hom functors, generalizes the Morita theorems.  As a corollary, we obtain a simple proof of a result of Gordon and Green characterizing graded Morita equivalences.  We also describe the twisted Hom functors that correspond to Zhang twists.

This section is somewhat technical because of the need to be careful with subtleties of notation; see Remark~\ref{rmk-notZ} for an indication of the potential pitfalls.  One helpful notational technique, borrowed from the author's previous paper \cite{S}, is to work with {\em $\zed$-algebras}, which are more general than graded rings.  
Recall that a $\zed$-algebra is a  ring $R$ without 1, with a $(\zed \times \zed)$-graded vector space decomposition $R  = \bigoplus_{i,j \in \zed} R_{ij}$, such that for all $i, j, l \in \zed$ we have $R_{ij} R_{jl} \subseteq R_{il}$, and if $j \neq j'$, then $R_{ij} R_{j' l} = 0$.  We require that the diagonal subrings $R_{ii}$ have units $1_i$ that act as a left identity on all  $R_{ij}$ and a right identity on all $R_{ji}$.  
If $R$ is a $\zed$-graded ring, the {\em $\zed$-algebra associated to $R$} is the ring
\[ \bbar{R} = \bigoplus_{i,j \in \zed} \bbar{R}_{ij},\]
where $\bbar{R}_{ij} = R_{j-i}$.  
Let $\rmodu \bbar{R}$ denote the category of unitary right $\bbar{R}$-modules; that is, modules $M$ such that $M \bbar{R} = M$.  Then it is easy to see that the categories $\rgr R$ and $\rmodu \bbar{R}$ are isomorphic, and we will identify them throughout.

We recall some terminology and results from \cite{S}.  Recall that a $\zed$-algebra $E$ is called {\em principal} if there is a $k$-algebra automorphism $\beta$ of $E$ that maps $E_{ij} \to E_{i+1,j+1}$ for all $i, j \in \zed$; $\beta$ is called a {\em principal automorphism} of $E$.   We recall:

\begin{proposition}\label{prop-principal}\cite[Proposition~3.3]{S}
Let $E$ be a $\zed$-algebra.  Then $E$ is principal if and only if there is a graded ring $B$ such that $E \cong \bbar{B}$ via a degree-preserving map. \qed
\end{proposition}

We review the constructions in Proposition~\ref{prop-principal}.  Suppose that $E$ is principal, and let $\gamma$ be a principal automorphism of $E$.  Then we may construct  a ring $E^\gamma$, which we call the {\em compression of $E$ by $\gamma$.}  As a graded vector space, $E^\gamma = E_{0*} =  \bigoplus_{i \in \zed} E_{0i}$.  The multiplication $\star$ on $E^\gamma$ is defined as follows:  if $ a \in (E^\gamma)_i = E_{0i}$ and $b \in (E^\gamma)_j = E_{0j}$, then
\[ a \star b = a \gamma^i(b) \in E_{0i} E_{i, i+j} \subseteq E_{0, i+j} = (E^\gamma)_{i+j}.\]
The proof of \cite[Proposition~3.3]{S} shows that  $\bbar{E^\gamma} \cong E$.   
Conversely, if  $S$ is a graded ring, then $\bbar{S}$ has a {\em  canonical principal automorphism}
 $\alpha$, given by the identifications  $\bbar{S}_{ij} = S_{j-i} = \bbar{S}_{i+1,j+1}$.

Suppose that $M = \bigoplus_{i,j \in \zed} M_{ij}$ is a {\em bigraded} right $R$-module --- that is, each $M_{i*} = \bigoplus_{j \in \zed} M_{ij}$ is a graded $R$-submodule of $M$ and $M \cong \bigoplus_{i \in \zed} M_{i*}$.  Assume for simplicity that each $M_{i*}$ is finitely generated.  Then we define the {\em endomorphism $\zed$-algebra of $M$} to be
\[ \Hfunct{R}{M}{M} = \bigoplus_{i, j \in \zed} \Hfunct{R}{M}{M}_{ij},\]
where each $\Hfunct{R}{M}{M}_{ij} = \hom_R(M_{j*}, M_{i*})$.  
Further,  for any graded $R$-module $N = \bigoplus_{i \in \zed} N_i$ there is a graded object 
\[ \Hfunct{R}{M}{N} = \bigoplus_{j \in \zed} \hom_R(M_{j*}, N).\]
Let $H = \Hfunct{R}{M}{M}$.  Then $\Hfunct{R}{M}{N}$ is naturally a right $H$-module, since $M$ has a left $H$-action.   If there is a graded ring $S$ with an isomorphism $\bbar{S} \cong H$, then $\Hfunct{R}{M}{\blank}$ may be regarded as a functor from $\rgr R$ to $\rgr S$.   

We recall from \cite{S} the following characterization of graded equivalences:
\begin{proposition}\label{prop-grMor}\cite[Proposition~2.1]{S}
Let $R$ and $S$ be graded rings.  Then $\rgr R \simeq \rgr S$ if and only if there is a bigraded right $R$-module $M_R = \bigoplus_{i , j \in \zed} M_{ij}$ such that 

$(1)$ $M_R$ is a projective generator for $\rgr R$ with each $M_{i*} = \bigoplus_{j \in \zed} M_{ij}$ finitely generated; 

$(2)$ the $\zed$-algebras $\bbar{S}$ and $\Hfunct{R}{M}{M}$ are isomorphic via a degree-preserving map.  

Further, if $M$ is as above, then $\Hfunct{R}{M}{\blank}: \rgr R \to \rgr S$ and $\blank \otimes_{\bbar{S}} M: \rgr S \to \rgr R$ are quasi-inverse equivalences; and if $\Psi: \rgr S \to \rgr R$ is an equivalence of categories, then $M = \Psi(\bbar{S}_S)$ satisfies (1) and (2), and $\Psi \cong \blank \otimes_{\bbar{S}} M$. \qed
\end{proposition}

Let $R$, $S$, and $M$ satisfy (1) and (2) of Proposition~\ref{prop-grMor}; let $H = \Hfunct{R}{M}{M}$.  Then $H \cong \bbar{S}$ is principal, with a principal automorphism $\beta$ induced  from the  canonical principal automorphism of $\bbar{S}$.  Thus $H^\beta \cong S$, and this isomorphism gives the equivalence between $\rgr R$ and $\rgr S$.
 That is, if $N = \bigoplus_{i \in \zed} N_i$ is a graded $R$-module, the $S$-action on $\Hfunct{R}{M}{N}$ is given by
\beq\label{mult1}
fs = f \circ \beta^j(s)
\eeq
for $f \in \Hfunct{R}{M}{N}_j = \hom_R(M_{j*}, N)$ and $s \in S_i \cong (H^\beta)_i = H_{0i}$.

We now  define twisted endomorphism rings.
Let $R$ be a graded ring, let $\mathcal{F}$ be an autoequivalence of $\rgr R$, and let $P$ be  a finitely generated graded projective $R$-module.  We say that $\mathcal{F}$ is {\em $P$-generative} if the modules $\{ \mathcal{F}^i P \}_{i \in \zed}$ generate $\rgr R$; more generally $\mathcal{F}$ is {\em generative} if it is $P$-generative for some finitely generated graded projective $P$.  It is easy to see that if $\Phi: \rgr R \to \rgr S$ is a category equivalence, then the pullback $\Phi^* \shift_S$ is generative in $\rgr R$.  The next proposition shows that this process also works in reverse:  there is a natural way to obtain a  ring graded equivalent to $R$ from a generative autoequivalence of $\rgr R$.

We pause  to discuss a subtlety of notation.  If $\mathcal{F}$ is an autoequivalence of $\rgr R$ that is  not an automorphism, we must be careful about defining the powers $\mathcal{F}^i$ for $i<0$.   Let $\F_{-1}$ be a quasi-inverse to $\mathcal{F}$.  If $i<0$, we define $\F^i = (\F_{-1})^{-i}$; we define $\F^0 = \id_{\rgr R}$.

In fact, we will work in even more generality.  Let  $\F_{\bullet} = \{ \F_i \}_{i \in \zed}$ be a sequence of autoequivalences of $\rgr R$.  We will say that $\F_{\bullet}$ is a {\em $\zed$-sequence} if: 

(1) $\mathcal{F}_0 = \id_{\rgr R}$;

(2) There are natural {\em compatibility isomorphisms} $\eta^{ij}: \F_{i+j} \to \F_i \F_j$ for all $i, j \in \zed$ that satisfy:
\begin{equation} \label{cond1}
\mbox{For all $j$, the maps $\eta^{0j}, \eta^{j0} : \mathcal{F}_j \to \mathcal{F}_j$ are the identity.}
\end{equation}
and 
\begin{equation}\label{cond2}
\begin{split}
\mbox{For all $i, j, l \in \zed$, the following diagram commutes:} \\
\xymatrix{
\F_{i+j+l} \ar[rr]^{\eta^{i, j+l}} \ar[d]_{\eta^{i+j, l}} && \F_i \F_{j+l} \ar[d]^{\F_i(\eta^{jl})}\\
\F_{i+j} \F_l \ar[rr]_{ \eta^{ij} \F_l	} && \F_i \F_j \F_l.}
\end{split}
\end{equation}

For example, for any autoequivalence $\F$, the sequence $\{ \F^j \}_{j \in \zed}$ is a $\zed$-sequence.  For the  compatibility isomorphisms $\eta^{ij} : \F^{i+j} \to \F^i \F^j$, let 
\[ \eta^{-1,1}:\id_{\rgr R} \to \F^{-1} \circ \mathcal{F}\]
be any natural isomorphism.   Then for any $N \in \rgr R$, define
\[ \eta^{1,-1}_N: N \to \F \F^{-1}N\]
to be the element of $\hom_R(N, \F \F^{-1} N)$ corresponding under $\F^{-1}$ to
\[ \eta^{-1,1}_{\F^{-1} N} : \F^{-1} N \to \F^{-1} \F \F^{-1} N.\]
Then \eqref{cond2} is satisfied  for $(i, j, l) = (\pm 1, \mp 1, \pm 1)$, and there is a unique way to define $\eta^{ij}$ for arbitrary $i, j$ so that \eqref{cond2} holds for all $i, j, l$.

\begin{proposition}\label{prop-shift}
Let $R$ be a graded ring and let $\mathcal{F}$ be an autoequivalence of $\rgr R$.  Then $\mathcal{F}$ is generative if and only if there are a graded ring $S$ and an equivalence of categories  $\Phi: \rgr R \to \rgr S$  such that $\mathcal{F} \cong \Phi^* \shift_S$.
\end{proposition}

\begin{proof}
Given an equivalence $\Phi: \rgr R \to \rgr S$ with quasi-inverse $\Psi$, by Proposition~\ref{prop-grMor}  the $R$-module $\Psi(\bbar{S}_S) \cong \bigoplus_{i \in \zed} (\Phi^* \shift_S)^i (\Psi S)$ generates $\rgr R$, so the autoequivalence $\Phi^*\shift_S$ is generative. 

Conversely, suppose that $\F$ is a $P$-generative autoequivalence of $\rgr R$ and that $\F_\bullet$ is a $\zed$-sequence with compatibility isomorphisms $\eta^{ij}$, such that $\F_1 = \F$.   Let $M = \bigoplus_{i \in \zed} \mathcal{F}_i P$, and let $H = \Hfunct{R}{M}{M}$.  The autoequivalence $\F$ and the maps $\eta^{ij}$ induce a principal automorphism $\beta$ of $H$ as follows:  if $s \in H_{ki} = \hom_R(\F_i P, \F_k P)$, define $\beta(s)$ to be the composition
\[ \xymatrix{
\mathcal{F}_{i+1}P \ar[r]^{\eta^{1i}}	& \mathcal{F} \mathcal{F}_i P \ar[r]^{\mathcal{F}(s)} 
	& \mathcal{F} \mathcal{F}_k P \ar[r]^{(\eta^{1k})^{-1}} & \mathcal{F}_{k+1} P	} \]
as an element of $H_{k+1, i+1} = \hom_R(\F_{i+1}P, \F_{k+1}P)$.
The compatibility condition \eqref{cond2} ensures that we have
\beq\label{alphan}
 \beta^j(s) = (\eta^{jk})^{-1} \circ \mathcal{F}_j (s) \circ \eta^{ji}
 \eeq
for all $j \in \zed$.  

Let $S$ be the compressed ring $H^{\beta}$.  Then by Propositions~\ref{prop-principal} and \ref{prop-grMor}, the functor $\Hfunct{R}{M}{\blank}$ is an equivalence between $\rgr R$ and $\rgr S$.    We need to show that $\mathcal{F}$ is isomorphic to the pullback of the the shift functor in $\rgr S$, or, equivalently, that 
$\mathcal{S}_S \circ \Hfunct{R}{M}{\blank} \cong \Hfunct{R}{M}{\blank} \circ \mathcal{F}$.  

Let $N \in \rgr R$.  
Define a map $\phi_N: \shift_S(\Hfunct{R}{M}{N}) \to \Hfunct{R}{M}{\mathcal{F}N}$ as follows:  if  $n \in (\shift_S \Hfunct{R}{M}{N})_j = \Hfunct{R}{M}{N}_{j-1} = \hom_R(\mathcal{F}_{j-1}P, N)$, let $\phi_N(n)$ be the composition 
\[ \xymatrix{
\phi_N(n):  \mathcal{F}_j P \ar[rr]^{\eta^{1,j-1}}	&& \mathcal{F} \mathcal{F}_{j-1} P
	\ar[rr]^{\mathcal{F}(n)}	&& \mathcal{F}N.	}\]
Clearly $\phi_N$ is an isomorphism of graded $k$-vector spaces; we verify that it  is an $S$-module map.

Let $s \in S_i = \hom_R(\F_i P, P)$ and let $n \in \shift_S(\Hfunct{R}{M}{N})_j = \Hfunct{R}{M}{N}_{j-1}$.  We need  to prove that $\phi_N(ns) = \phi_N(n) s$.  From \eqref{mult1} we have that the $S$-action on $\Hfunct{R}{M}{N}$ is defined  by $ns = n \circ \beta^{j-1}(s)$.  Therefore, $\phi_N(ns)$ is given by the diagram
\[ \xymatrix{
 \F_{i+j} P \ar[rd]_{\eta^{1, i+j-1}}&&&& \F \F_{j-1} P \ar[r]^(.6){\F(n)} 	& \F N. \\
 & \F \F_{i+j-1} P \ar @/^15pt/ [rrru]^{\F(\beta^{j-1}(s))} \ar[rr]_{\F(\eta^{j-1,i})}
	&& \F \F_{j-1} \F_iP  \ar[ru]_{\F \F_{j-1} (s)}}
	\]
On the other hand, the $S$-action on  $\Hfunct{R}{M}{\F N}$ is given by $\phi_N(n) s = \phi_N(n) \circ \beta^j(s)$, which corresponds to the diagram
\[ \xymatrix{
\F_{i+j}P \ar[r]_{\eta^{ji}} \ar @/^15pt/ [rr]^{\beta^j(s)}
	& \F_j \F_i P \ar[r]_{\F_j(s)}
	& \F_j P \ar[r]_{\eta^{1,j-1}} \ar @/^15pt/ [rr]^{\phi_N(n)}
	& \F \F_{j-1} P \ar[r]_(.6){\F(n)}
	& \F N. }\]
  So we need to prove the commutativity of 
\[ \xymatrix{
\mathcal{F}_{i+j}P \ar[r]^{\eta^{ji}}  \ar[d]^>>>>>{\eta^{1,i+j-1}}
	& \mathcal{F}_j \mathcal{F}_i P \ar[r]^{\mathcal{F}_j(s)} 
	& \mathcal{F}_j P \ar[r]^{ \eta^{1,j-1}} 
	& \mathcal{F} \mathcal{F}_{j-1} P \ar[r]^(.6){\mathcal{F}(n)}
	& \mathcal{F} N. \\
	 \mathcal{F} \mathcal{F}_{i+j-1} P \ar[rr]^{\mathcal{F}(\eta^{j-1,i})}
	&& \mathcal{F} \mathcal{F}_{j-1} \mathcal{F}_i P \ar[ru]_{\mathcal{F}\F_{j-1}(s)} } \]
  This follows from the compatibility condition \eqref{cond2} and the naturality of $\eta^{1,j-1}$.  

The naturality of $\phi$ is even easier.  Let $f: N \to N'$ be a homomorphism of graded $R$-modules and let $n \in \shift_S(\Hfunct{R}{M}{N})_j$ as above.  We leave it to the reader to verify that the naturality of the isomorphism $\phi$ reduces to establishing the commutativity of 
\[ \xymatrix{
\mathcal{F}_jP \ar[rr]^(0.4){\eta^{1,j-1}}
	&& \mathcal{F} \mathcal{F}_{j-1} P \ar[r]^(.6){\mathcal{F}(n)} \ar[rd]_{\mathcal{F}(f \circ n)}
	& \mathcal{F}N \ar[d]^{\mathcal{F}(f)} \\
	&&	&	\mathcal{F}N'.	} \]
This is immediate from functoriality of $\F$.
\end{proof}

\begin{defn}\label{defn-twist}
We call the ring $S = H^{\beta} = \bigoplus_{n \in \zed} \hom_R(\F_nP, P)$ from the above proof the 
 {\em $\F$-twisted endomorphism ring of $P$}, and we write it $\End^{\F}_R(P)$.
   We call the functor
\[ \Hfunct{R}{\bigoplus_{n \in \zed} \mathcal{F}_n P}{\blank}: \rgr R \to \rgr \End^{\F}_R(P)\]
the {\em $\F$-twisted Hom functor} of $P$, and we write it $\Htwist{\F}{R}{P}{\blank}$.  
\end{defn}

Thus we have proved:

\begin{proposition}\label{prop-pullback}
Let $R$ be a graded ring.  Let $P$ be a finitely generated projective graded $R$-module and let $\F$ be a $P$-generative autoequivalence of $\rgr R$.  Let  $S = \End^{\mathcal{F}}_R(P)$ and let $\Phi = \Htwist{\F}{R}{P}{\blank}: \rgr R \to \rgr S$.    Then $\Phi$ is an equivalence of categories, and $\F \cong \Phi^* \shift_S$.  
\qed
\end{proposition}

\begin{remark}\label{rmk-opp}
We comment that if $P$ is a finitely generated graded projective right $R$-module and $\sh{F}$ is a $P$-generative autoequivalence of $\rgr R$,  one may also use $\F$, or more properly the principal automorphism $\beta$ used in the proof of Proposition~\ref{prop-shift}, to define a natural ring structure on 
\[ \bigoplus_{n \in \zed} \hom_R(P, \sh{F}^{-n}P).\]
This ring is clearly isomorphic to $\End^{\F}_R(P)$ as we have defined it above.
\end{remark}

It is easy to see that the $\shift_R$-twisted endomorphism ring
$\End^{\shift_R}_R(P)$ is isomorphic to the ungraded endomorphism ring $\End_R(P)$, and that
$\Htwist{\shift_R}{R}{P}{\blank} \cong \Hom_R(P, \blank)$.  Thus twisted endomorphism rings generalize classical endomorphism rings.  
We also note that $\End^{\F}_R(P)$ and $\Htwist{\F}{R}{P}{\blank}$ may be defined for any graded $R$-module $P$, although by Proposition~\ref{prop-grMor} they will not give an equivalence of categories unless $P$ is finitely generated projective and $\F$ is $P$-generative.

A priori, the ring $\End^\F_R(P)$ and the functor $\Htwist{\F}{R}{P}{\blank}$ may depend on the full $\zed$-sequence $\F_\bullet$.  However, in fact they do not.  To see this, let $\F_\bullet$ be a $P$-generative $\zed$-sequence with compatibility isomorphisms $\eta^{ij}$.    Let  $\sh{B} \cong \F_1$, and  let $\sh{B}^{-1}= \F_{-1}$.  Then $\sh{B}^{-1}$ is a quasi-inverse for $\sh{B}$; let $\beta^{-1,1}: \id_{\rgr R} \to \sh{B}^{-1} \sh{B}$ be induced from $\eta^{-1,1}$.  Let $\{ \beta^{ij} \}$ be the (induced) compatibility isomorphisms on $\{\sh{B}^n\}$ satisfying \eqref{cond2}, as discussed before Proposition~\ref{prop-shift}.  Then for all $j \in \zed$ there are unique induced natural isomorphisms $\zeta_j: \sh{B}^j \to \F_j$, such that the diagram
\[ \xymatrix{
\mathcal{B}^{i+j} \ar[r]^{\zeta_{i+j}}	\ar[d]_{\beta^{ij}} & \mathcal{F}_{i+j} \ar[d]^{\eta^{ij}} \\
\mathcal{B}^i \mathcal{B}^j \ar[r]_{\zeta_i \zeta_j}
	& 	\mathcal{F}_i \mathcal{F}_j } \]
commutes for all $i, j \in \zed$.   This commutativity implies that the twisted endomorphism rings $\End^{\F_1}_R(P)$ and $\End^{\sh{B}}_R(P)$ are isomorphic, and that 
$\Htwist{\F_1}{R}{P}{\blank} \cong \Htwist{\sh{B}}{R}{P}{\blank}$.  We leave the tedious but routine verifications to the reader.  Thus we have:
\begin{proposition}\label{prop-welldef}
Let $\F$ be a $P$-generative autoequivalence of $\rgr R$.  Then the $\F$-twisted endomorphism ring and $\F$-twisted Hom functor of $P$ are well-defined, and depend only on the equivalence class of $\F$ in $\Pic(\rgr R)$. \qed
\end{proposition}

We are ready to reframe Proposition~\ref{prop-grMor} in terms of autoequivalences of $\rgr R$.  
\begin{theorem}\label{thm-grMor-pullback}
Let $R$ and $S$ be graded rings.  Then $\rgr R \simeq \rgr S$ if and only if there are

$(1)$ a finitely generated graded projective module $P$ and a $P$-generative autoequivalence $\F$ of $\rgr R$ such that 

$(2)$ $S \cong \End^{\F}_R(P)$ via a degree-preserving map.

Further, if $\Phi: \rgr R \to \rgr S$ is an equivalence  of categories with quasi-inverse $\Psi$, then $\F = \Phi^* \shift_S$ and $P = \Psi S$ satisfy (1) and (2), and $\Phi \cong \Htwist{\F}{R}{P}{\blank}$.
\end{theorem}
\begin{proof}
One direction is Proposition~\ref{prop-pullback}.  For the other, suppose that  $\Phi: \rgr R \to \rgr S$ is an equivalence with quasi-inverse $\Psi$.  
Let $\sh{F}  = \Phi^* \shift_S$, and let $P = \Psi S_S$.  
By Proposition~\ref{prop-grMor}, we have that $\Phi \cong \Hfunct{R}{M}{\blank}$, where 
\[M = \Psi(\bbar{S}_S) \cong \Psi(\bigoplus_{n \in \zed} \shift^n_S(\Phi P)),\]
 and $S \cong \Hfunct{R}{M}{M}^\beta$, where $\beta$ is the principal automorphism of $\Hfunct{R}{M}{M}$ induced from the identifications
\[ \hom_R(\Psi \shift_S^j \Phi P, \Psi \shift_S^i \Phi P) \cong S_{j-i} \cong \hom_R(\Psi \shift_S^{j+1} \Phi P, \Psi \shift_S^{i+1} \Phi P).\]
That is, $\beta$ is given by the autoequivalence $\sh{F}$ of $\rgr R$ and by the compatibility isomorphisms on the $\zed$-sequence $\sh{F}_\bullet = \{  \Psi \shift_S^j \Phi \} = \{ \Phi^* \shift_S^j \}$.   Since $\sh{F}_1 = \sh{F}$, by Proposition~\ref{prop-welldef} we have that $S \cong \End^{\sh{F}}_R(P)$ and that $\Phi \cong \Htwist{\sh{F}}{R}{P}{\blank}$.  
\end{proof}

We comment that although Proposition~\ref{prop-grMor} and Theorem~\ref{thm-grMor-pullback} are  closely related,  in many respects Theorem~\ref{thm-grMor-pullback} is superior.  Theorem~\ref{thm-grMor-pullback} is a much more functorial way of constructing equivalences, uses only one finitely generated graded $R$-module,  and more compactly and clearly generalizes the ungraded Morita theorems.    Theorem~\ref{ithm1}, which is the main application in this paper of Theorem~\ref{thm-grMor-pullback}, gives one indication of the power of our approach. 

 In the remainder of this section and in Section~\ref{sec-genPic}, we give some general applications of Theorem~\ref{thm-grMor-pullback}.  We first note that  Theorem~\ref{thm-grMor-pullback}, when combined with the seemingly purely formal Proposition~\ref{prop-welldef},  has surprising power.  For example, we immediately obtain the following result of Gordon and Green \cite[Corollary~5.2, Proposition~5.3]{GG}: 
\begin{corollary} \label{cor-Morita}
Let $R$ and $S$ be graded rings, and let $\Phi: \rgr R \to \rgr S$ be an equivalence of categories.  Then $\Phi$ is a graded Morita equivalence if and only if $\Phi$ commutes with shifting in the sense that $\Phi^* \shift_S \cong \shift_R$.
\end{corollary}
\begin{proof}
If $\Phi$ is a  graded Morita equivalence, then we have $\Phi \cong  \Hom_R(P, \blank) = \Htwist{\shift_R}{R}{P}{\blank}$ for some graded module $P$ that is a progenerator in $\rmod R$, and $S \cong \End_R(P) = \End^{\shift_R}_R(P)$.  By Proposition~\ref{prop-pullback}, $\shift_R \cong \Phi^*\shift_{S}$.

Conversely, suppose that $\Phi^* \shift_S \cong \shift_R$.    Let $\Psi$ be the quasi-inverse to $\Phi$,  let $P = \Psi(S)$, and let $\F = \Phi^* \shift_S$.  Since $\F \cong \shift_R$, by Proposition~\ref{prop-welldef} $\End_R(P) = \End^{\shift_R}_R(P)$ is isomorphic to $\End^{\F}_R(P)$, which by Theorem~\ref{thm-grMor-pullback} is isomorphic to $S$.  Furthermore, applying the same results we see that 
\[\Phi \cong \Htwist{\F}{R}{P}{\blank} \cong \Htwist{\shift_R}{R}{P}{\blank} = \Hom_R(P, \blank)\]
 is a graded Morita equivalence.
\end{proof}

\begin{remark}\label{rmk-notZ}
Versions of many the results of this section also hold for rings graded by arbitrary groups.  Suppose that $G'$ and $G$ are groups, and that $R$ is a $G'$-graded ring.  The idea is to work with {\em $G$-sequences} $\F_\bullet = \{ \F_g \}_{g\in G}$ of autoequivalences of $\rgr R$, defined as sequences with appropriate compatibility isomorphisms
$\eta^{f,g}: \F_{fg} \to \F_f \F_g$.  Given a $G$-sequence $\F_\bullet$ and a graded $R$-module $P$, we may construct a $G$-graded ring $\End^{\F_{\bullet}}_R(P)$ and a functor $\Htwist{\F_\bullet}{R}{P}{\blank}$, and  suitably modified versions of Proposition~\ref{prop-shift} and Theorem~\ref{thm-grMor-pullback} still hold.  However, Proposition~\ref{prop-welldef} does not hold for arbitrary grading groups:  it is possible to have two $G$-sequences $\sh{B}_\bullet$ and $\sh{C}_\bullet$ such that for all $g \in G$ we have $\sh{B}_g \cong \sh{C}_g$, but the functors $\Htwist{\sh{B}_\bullet}{R}{P}{\blank}$ and $\Htwist{\sh{C}_\bullet}{R}{P}{\blank}$ are not isomorphic.  

This means that Corollary~\ref{cor-Morita} does not hold for rings graded by arbitrary groups.  
  In  \cite[Example~4.8]{BR2} the authors construct a ring $R$ graded by the symmetric group $\Sigma_3$ with the property  that the set of autoequivalences $\Phi$ of $\rgr R$ such that $\Phi^* \shift_R^g \cong \shift_R^g$ for all $g \in \Sigma_3$ is strictly larger than the set of graded Morita equivalences from $\rgr R$ to $\rgr R$.   We note that part of the proof of this statement is contained in \cite{dRErr}.
  \end{remark}

For $\zed$-graded rings, 
Corollary~\ref{cor-Morita} is extremely powerful, since it provides an easy way to check if two rings in the graded equivalence class of $R$ are Morita equivalent.  The following application will be particularly helpful for us:

\begin{proposition}\label{prop-Morita-conjugate}
Let $R$ be a graded ring, let $P$ and $P'$ be finitely generated graded projective $R$-modules, and suppose that $\mathcal{F}$ and $\F'$ are  autoequivalences of $\rgr R$ such that $\F$ is $P$-generative and $\mathcal{F}'$ is $P'$-generative.
  Then $\End^{\mathcal{F}}_R(P)$ and $\End^{\mathcal{F}'}_R(P')$ are graded Morita equivalent if and only if $\F$ and $\F'$ are conjugate in $\Pic(\rgr R)$.
\end{proposition}
\begin{proof}
Let $T = \End^{\mathcal{F}}_R(P)$ and let $T' = \End^{\mathcal{F}'}_R(P')$.  
Let
$\Phi = \Htwist{\mathcal{F}}{R}{P}{\blank}: \rgr R \to \rgr T$; likewise let
$\Phi' = \Htwist{\mathcal{F}'}{R}{P'}{\blank}: \rgr R \to \rgr T'$.  By Theorem~\ref{thm-grMor-pullback}, $\Phi$ and $\Phi'$ are equivalences of categories; let $\Psi$ and $\Psi'$ be their respective quasi-inverses.

First assume that $\F$ and $\F'$ are conjugate in $\Pic (\rgr R)$, so there is an autoequivalence $\iota$ of $\rgr R$ such that $\iota \F' \cong \F \iota$.  Let $\iota^{-1}$ be a quasi-inverse to $\iota$.  
Define $\Theta = \Phi \iota \Psi':  \rgr T' \to \rgr T$; note that $\Phi' \iota^{-1} \Psi$ is a quasi-inverse to $\Theta$.  

We apply Corollary~\ref{cor-Morita} and compute $\Theta^* \shift_T$.  We have:
\[ \Theta^* \shift_T = \Phi' \iota^{-1} \Psi \shift_T \Phi \iota \Psi' = (\Phi')_* (\iota^{-1} \Phi^* (\shift_T) \iota).\]
Since by Proposition~\ref{prop-pullback} we have $\Phi^* \shift_T \cong \mathcal{F}$ and 
$(\Phi')_* \mathcal{F}' \cong (\Phi')_*( \Phi')^* \shift_{T'} \cong \shift_{T'}$, we see that
\[\Theta^* \shift_T = \Phi'_* (\iota^{-1} \Phi^* (\shift_T) \iota) \cong \Phi'_*(\iota^{-1} \mathcal{F} \iota) \cong \Phi'_*(\mathcal{F}') \cong \shift_{T'}.\]
Thus by Corollary~\ref{cor-Morita}, $\Theta$ is a graded Morita equivalence.

Conversely, suppose that $\Theta: \rgr T' \to \rgr T$ is a graded Morita equivalence.  By Corollary~\ref{cor-Morita}, we have $\Theta \shift_{T'} \cong \shift_T \Theta$.  Let $\iota = \Psi \Theta \Phi' \in \Pic(\rgr R)$.  Then  we have that 
\[
\iota \F' \cong \iota \Psi' \shift_{T'} \Phi' \cong \Psi \Theta \Phi' \Psi' \shift_{T'} \Phi' 
 \cong \Psi \Theta \shift_{T'} \Phi' \cong \Psi \shift_T \Theta \Phi' \cong \Psi \shift_T \Phi \Psi \Theta \Phi' = \F \iota.
\]
Thus $\F$ and $\F'$ are conjugate in $\Pic(\rgr R)$.  
\end{proof}

\section{The Picard group of a graded module category}\label{sec-genPic}
If $S$ is an ungraded ring, there is a classical map from $\Aut S \to \Pic S$.  Beattie and del R\'io \cite{BR} have observed  that this extends, using $\zed$-algebras, to graded categories.  That is,  given a graded ring $R$, there is a group homomorphism 
\begin{align}\label{Aut-Pic}\begin{split}
\pi: \Aut \bbar{R} 	& \to \Pic(\rgr R) \\
 \gamma			& \mapsto (\blank)_{\gamma}. \end{split} \end{align}
We review the construction of the map \eqref{Aut-Pic}.  Let $\gamma$ be an automorphism of the $\zed$-algebra $\bbar{R}$, and  let $M$ be an $\bbar{R}$-module.  Then there is a twisted module $M_{\gamma}$ that equals $M$ as a  vector space, with multiplication $\circ^\gamma$ defined by  
\[m \circ^\gamma r = m \gamma^{-1} (r).\]
  The functor $(\blank)_\gamma$ is isomorphic to $\blank \otimes_{\bbar{R}} \bbar{R}_{\gamma}$, and clearly if $M_{\bbar{R}}$ is unitary, so is $M_\gamma$.   Since unitary $\bbar{R}$-modules and graded $R$-modules are the same, we may regard $\pi(\gamma) =  (\blank)_\gamma$ as an element of  $\Pic(\rgr R)$; it is a quick computation that $\pi(\alpha\beta) = \pi(\alpha) \circ \pi(\beta)$.    

As an example, let $\alpha$ be the canonical principal automorphism of $\bbar{R}$.  Then $(\blank)_\alpha \cong \shift_R$:  if $M$ is a graded $R$-module, then $(M_\alpha)_i = M \circ^\alpha 1_i = M \cdot \alpha^{-1}(1_i) = M \cdot 1_{i-1} = M_{i-1}$.  This isomorphism is easily seen to be natural.

In this section, we explore applications of the map \eqref{Aut-Pic} to the study of equivalences of graded module categories.  We derive implications for the equivalences known as Zhang twists, and also study the kernel of \eqref{Aut-Pic}.

To begin, we move beyond graded Morita equivalences to investigate the other well-known examples of graded equivalences: the equivalences induced by  {\em twisting systems}, as defined by Zhang \cite{Zh}.  We recall that a twisting system  is a set $\tau = \{ \tau_n\} $ of graded $k$-vector space automorphisms  of $R$ satisfying relations
\[ \tau_{n}(a) \tau_{n+m} (b) = \tau_n(a \tau_m (b))\]
for all $n, m, l \in \zed$, $a \in R_m$, and $b \in R_l$.    Given a ring $R$ and a twisting system $\tau$ on $R$, we may form a {\em twisted algebra}  $R^\tau$, with multiplication $\star$ defined by $a \star b = a \tau_m(b)$ for $a \in R_m$ and $b \in R_l$.  These twisted algebras are often referred to as {\em Zhang twists}.    In \cite[Corollary~4.4]{S}, the author characterized Zhang twists as follows:
\begin{theorem}\label{thm-zhangchar}
Let $R$ and $S$ be graded rings.  Then the following are equivalent:

$(1)$ $S$ is isomorphic to a Zhang twist of $R$ via a degree-preserving map.

$(2)$ The $\zed$-algebras $\bbar{R}$ and $\bbar{S}$ are isomorphic via a degree-preserving map.

$(3)$ There is a principal automorphism $\beta$ of $\bbar{R}$ so that $S \cong \bbar{R}^\beta$ via  a degree-preserving map.

$(4)$  There is an equivalence $\Phi: \rgr R \to \rgr S$ such that  for all $n \in \zed$, $\Phi(R \ang{n}) \cong S \ang{n}$. \qed
\end{theorem}

We call a functor $\Phi$ satisfying Theorem~\ref{thm-zhangchar}(4) a {\em twist functor.}   We will prove a result analogous to Corollary~\ref{cor-Morita}, characterizing twist functors $\Phi: \rgr R\to \rgr S$  in terms of the pullbacks $\Phi^* \shift_S$.  

We first make a technical comment:  all  twist functors may all be written in a standard form.  Suppose that we are given a graded ring $R$ and  a principal automorphism $\beta$ of $\bbar{R}$.    Let $\alpha$ be the canonical principal automorphism of $\bbar{R}$.    Any graded $R$-module is also an $\bbar{R}$-module and so an $\bbar{R}^\beta$-module.  Let $\circ_R$ denote the $R$-action on $M$, and let $\star$ denote the $\bbar{R}^\beta$-action.  Then the $\bbar{R}^\beta$-action on $M$ is given as follows:  for any $m \in M_i$ and $s \in \bbar{R}^\beta_j$ we have
\beq\label{twist1}m \star s = m \circ_R \alpha^{-i} \beta^i(s) \eeq
 (note that $\alpha^{-i} \beta^i(s) \in \bbar{R}_{0j} = R_j$).  This gives an equivalence $\rgr R \to \rgr \bbar{R}^\beta$, which we call a {\em standard twist functor}.

We leave to the reader the verification of the following:

\begin{lemma}\label{lem-zhangfuncts}
Let $R$ and $S$ be graded rings, and let $\Phi: \rgr R \to \rgr S$ be a twist functor.  Let $\beta$ be the principal automorphism of $\bbar{R}$ induced by $\shift_S$ and $\Phi$, and let $R' = \bbar{R}^\beta$.   Let $Z: \rgr R \to \rgr R'$ be the standard twist functor.  Let $\phi: S \to R'$ be the induced ring isomorphism.  Then the diagram
\[ \xymatrix{
\rgr R \ar[r]^{Z} \ar[rd]_{\Phi}	& \rgr R' \\
	& \rgr S \ar[u]_{\phi}	}\]
commutes.   That is, without loss of generality any twist functor is standard. \qed
\end{lemma}

We now give the desired version of Corollary~\ref{cor-Morita} for twist functors.  
\begin{proposition}\label{prop-twists}
Let $R$ be a graded ring, and let $\F$ be an autoequivalence of $\rgr R$.  Then $\Htwist{\F}{R}{R}{\blank}$ is a twist functor if and only if $\F \cong ( \blank)_\beta$ for some principal automorphism $\beta$ of $\bbar{R}$.
\end{proposition}
\begin{proof} 
Let $\alpha$ denote the canonical principal automorphism of $\bbar{R}$.

$(\Rightarrow)$  Suppose that $\Phi = \Htwist{\F}{R}{R}{\blank}: \rgr R \to \rgr S$ is a twist functor.  By Proposition~\ref{prop-pullback}, $\F \cong \Phi^* \shift_S$.  Let $\Psi$ be a quasi-inverse for $\Phi$.  By Lemma~\ref{lem-zhangfuncts}, without loss of generality we may assume that $S = \bbar{R}^\beta$ for some principal automorphism $\beta$ of $\bbar{R}$, and that $\Phi$ and $\Psi$ are standard twist functors.   We will denote multiplication in $\rgr R$ by $\circ_R$; similarly for $\circ_S, \circ_{\bbar{R}}$.  
If $M$ is a graded $R$-module, then as a graded vector space $\Phi(M) = M$, with the $S$-module structure on $\Phi(M)$ given as follows:  if $m \in M_i$ and $s \in S_j = \bbar{R}_{0j} = R_j$, then by \eqref{twist1}
\beq \label{twistA} m \circ_S s = m \circ_R \alpha^{-i} \beta^i (s)\eeq
and likewise 
\beq \label{twistB} m \circ_R s = m \circ_S \beta^{-i} \alpha^i (s). \eeq

  We claim that $\Phi^* \shift_S \cong (\blank)_\beta$.  
Let $M$ be  a graded $R$-module.  Note that  $(\F M)_i = (\Psi \shift_S \Phi M)_i = M_{i-1} = M \circ_{\bbar{R}} \beta^{-1}(1_i) = (M_\beta)_i$.  We verify that $\F M$ and $M_\beta$ are isomorphic as $R$-modules.  
 Let $m \in (M_\beta)_{i} = (\F M)_{i} = M_{i-1}$, and let $r \in R_j$.   Let $\circ_1$ denote the $R$-action on $ M_\beta$, and let $\circ_2$ denote the $R$-action on $\F M$.   
Then we have that
\[ m \circ_1 r  = m \circ_{\bbar{R}} \beta^{-1} \alpha^i (r) = m \circ_R \alpha^{-(i-1)} \beta^{-1} \alpha^i (r).\]
On the other hand, since $m \in (\F M)_{i} = (\Psi \shift_S \Phi M)_{i}$, 
we have by \eqref{twistA} and \eqref{twistB} that
\begin{multline*}
m \circ_2 r = m \circ_S \beta^{-i} \alpha^{i} (r) = m \circ_R \alpha^{-(i-1)} \beta^{i-1} \beta^{-i} \alpha^{i} (r) \\
 = m \circ_R \alpha^{-(i-1)}\beta^{-1} \alpha^i (r) = m \circ_1 r.
 \end{multline*}
 Thus $\F M \cong M_\beta$; we leave to the reader the verification that this isomorphism is natural in $M$.

$(\Leftarrow).$  Suppose that $\F \cong (\blank)_\beta$, where $\beta$ is a principal automorphism of $\bbar{R}$.  Let  $S = \bbar{R}^{\beta}$.  As mentioned in the discussion after Proposition~\ref{prop-principal}, the proof of \cite[Proposition~3.3]{S} shows that $\bbar{R} \cong \bbar{S}$.  By Theorem~\ref{thm-zhangchar}, there is an equivalence of categories $\Phi: \rgr R \to \rgr S$, where $\Phi$ is a twist functor.   Let $\mathcal{G} = \Phi^* \shift_S$.  Then by Theorem~\ref{thm-grMor-pullback}, $\Phi \cong \Htwist{\mathcal{G}}{R}{R}{\blank}$.  But by the discussion above, we see that $\mathcal{G} \cong (\blank)_{\beta} \cong \F$; and by Proposition~\ref{prop-welldef}, $\Htwist{\mathcal{F}}{R}{R}{\blank}$ is isomorphic to $\Htwist{\sh{G}}{R}{R}{\blank}$ and so is a twist functor.
\end{proof}

\begin{corollary}\label{cor-ZhangMor}
Let $R$ and $S$ be graded rings.  An equivalence $\Phi: \rgr R \to \rgr S$ is a Zhang twist followed by a graded Morita equivalence if and only if $\Phi^* \shift_S \cong (\blank)_\gamma$ for some principal automorphism $\gamma$ of $\bbar{R}$. 
\end{corollary} 
\begin{proof}
The proof is similar to the proof of Corollary~\ref{cor-Morita}; since we will not use the result in the sequel, we omit the details.
\end{proof}

Let $R$ be a graded ring, and let $\alpha$ be the canonical principal automorphism of $\bbar{R}$.  Since we have seen that $(\blank)_\alpha \cong \shift_R$, we immediately have the following corollary:
\begin{corollary}\label{cor-automs}
Let $R$ and $S$ be graded equivalent graded rings, and let $\pi$ be the map defined in \eqref{Aut-Pic}.  If $\Phi: \rgr R \to \rgr S$ is 
either a graded Morita equivalence or a twist functor, then $\Phi^* \shift_S \in \im \pi$.
\qed
\end{corollary}

We now change focus from the image of \eqref{Aut-Pic} to its kernel.  
Beattie and del R\'io's original interest in the map \eqref{Aut-Pic} was to extend the 
 classical short exact sequence for an ungraded ring $S$
\[ \xymatrix{
0 \ar[r]	& \Inn S \ar[r] & 	\Aut S \ar[r] & \Pic S,	} \]
  to rings with local units, in particular to $\zed$-algebras, and to characterize the kernel of  \eqref{Aut-Pic}.  Let $R$ be a graded ring.  We will say that an automorphism $\gamma$ of $\bbar{R}$ is {\em inner} if $\gamma$ lies in the kernel of \eqref{Aut-Pic}.  The following is a special case of one of the results of \cite{BR}:
\begin{theorem}\label{thm-BR}\cite[Theorem~1.4]{BR}
Let $R$ be a graded ring, and let $\gamma$ be a graded automorphism of $\bbar{R}$ of degree 0.  Then $\gamma$ is inner if and only if for all $m,  n \in \zed$ there are $f_m \in \bbar{R}_{mm} = R_0$ and $g_n \in \bbar{R}_{nn} = R_0$ so that for all $w \in \bbar{R}_{mn}$, $\gamma(w) = f_m w g_n$. \qed
\end{theorem}

The image of \eqref{Aut-Pic} is almost never trivial,  since $\shift_R$ is not (usually) the identity in $\rgr R$. However, if all  degree 0 automorphisms of $\rgr R$ are inner, then we expect  the image of \eqref{Aut-Pic} to be small.  In this situation, there are strong consequences for $\rgr R$.
We first see that that there are no nontrivial Zhang twists of $R$.    

\begin{corollary}\label{cor-notwist}
Let $R$ be a graded ring such that all automorphisms of $\bbar{R}$ of degree 0 are inner.  If  $\Phi: \rgr R \to \rgr S$ is a twist functor, then $S \cong R$ and $\Phi$ is naturally isomorphic to $\id_{\rgr R}$.
\end{corollary}
\begin{proof}
Let $\Phi: \rgr R \to \rgr S$ be a twist functor; let $\F = \Phi^* \shift_S$.  By Proposition~\ref{prop-twists}, we know that $\F \cong  (\blank)_\beta$, for some  principal automorphism  $\beta$ of $\bbar{R}$.  Let $\alpha$ be the canonical principal automorphism of $\bbar{R}$; then $\alpha^{-1} \beta$ has degree 0 and so by assumption is inner.  Thus $\F \cong (\blank)_\beta\cong (\blank)_\alpha \cong \shift_R$.  By Theorem~\ref{thm-grMor-pullback}  $S \cong \End_{R}^{\shift_{R}}(R) \cong R$, and $\Phi \cong \Htwist{\shift_{R}}{R}{R}{\blank} \cong \id_{\rgr R}$.
\end{proof}

As a second consequence, we show that all autoequivalences of $\rgr R$ are determined by their action on (isomorphism classes of) $R$-modules.  Formally, let $\F$ and $\F'$ be autoequivalences of the category $\rgr R$.  We will say that  $\mathcal{F}$ and $\mathcal{F}'$ are {\em weakly isomorphic}, written $\mathcal{F} \weak \mathcal{F}'$, if for all finitely generated graded $R$-modules $M$, we have $\mathcal{F}(M) \cong \mathcal{F}'(M)$.   There is no assumption on the naturality of this isomorphism.  

\begin{corollary}\label{cor-noweak}
Let $R$ be a ring such that all automorphisms of $\bbar{R}$ of degree 0 are inner, and let $\F$ and $\F'$ be autoequivalences of ${\rgr R}$.  Then $\F \weak \F'$ if and only if $\F$ and $\F'$ are naturally isomorphic.
\end{corollary}
\begin{proof}
Suppose that $\F \weak \F'$.  Let $\mathcal{G} =  \F (\F')^{-1}$.  Then $\mathcal{G}$ is weakly isomorphic to $\id_{\rgr R}$, and so for all $n$, we have $\mathcal{G} (R\ang{n}) \cong R\ang{n}$.  That is, $\mathcal{G}: \rgr R \to \rgr R$ is a twist functor.  But now by Corollary~\ref{cor-notwist}, $\mathcal{G} \cong \id_{\rgr R}$ and so $\F \cong \F'$.
\end{proof}

\section{Graded modules over the Weyl algebra}\label{sec-category}
From this point on, we will work with the Weyl algebra and its graded module category.   
 Let $k$ be an algebraically closed  field of characteristic 0, and let  $A = k\{x,y\} /(xy-yx-1)$ be the (first) Weyl algebra over $k$.  We will grade $A$ by the Euler gradation, where  we put $\deg x = 1$ and $\deg y = -1$.    

The goal of the next two sections is to understand the graded Picard group of $A$.    In Section~\ref{sec-main} we will use these results and the methods of Sections~\ref{sec-general} and \ref{sec-genPic} to prove Theorem~\ref{ithm1}.

We fix notation, which will remain in force for the remainder of the paper.  We will let $z = xy$, so that $A_0 = k[z]$.  
Let $\sigma$ be the automorphism of $k[z]$ given by  $\sigma(z) = z+1$.  We will constantly use the fact that graded right $A$ modules are actually graded $(k[z], A)$-bimodules: since the grading on $A$ is given by commuting with $z$, we have that $\rho \in k[z]$ acts on a right $A$-module $M$ by
\[ \rho \cdot m = m \sigma^{-i}(\rho)\]
for any $m \in M_i$.  

Our first step in analyzing the category $\rgr A$ is to understand simple objects and the extensions between them.  There are very few of these:  if $M$ and $M'$ are graded simple modules, then $\Ext_A(M, M')$ is either 0 or $k$.    On the other hand, McConnell and Robson \cite{MR} have shown  that, while $\Ext_A(M,N)$ is always finite-dimensional for ungraded simple $A$-modules $M$ and $N$, there are simple $M$ and $N$ with infinitely many nonisomorphic extensions of $M$ by $N$.  This is one indication, that, as we will repeatedly see, the graded structure of $A$ is much more rigid than the ungraded structure.

Lemmas~\ref{lem-simple} and \ref{lem-Ext} are  elementary and presumably known, but   we have not been able to find these specific results in the literature.  

As $A$ is hereditary, we will use the abbreviation $\Ext_A$ for $\Ext^1_A$, and $\ext_A$ for $\ext^1_A$.  

\begin{lemma}\label{lem-simple}
Up to isomorphism, the graded simple $A$-modules are precisely:
\begin{itemize}
\item $X = A/xA$ and its shifts $X \ang{n}$ for $n \in \zed$;
\item $Y = A/yA \ang{1}$ and its shifts $Y \ang{n}$;
\item $M_\lambda = A / (z+\lambda) A$ for $\lambda \in k \smallsetminus \zed$.
\end{itemize}
\end{lemma}
\begin{proof}
The graded $A$-modules are precisely those that decompose as a sum of weighted subspaces over $k[z]$.   By \cite[Theorem~3.2]{Bav}, up to ungraded isomorphism the simple $k[z]$-weighted $A$-modules are
\begin{itemize}
\item $A/xA$;
\item $A/yA$;
\item One module $M_\lambda$ for each coset of $(k \smallsetminus \zed) / \zed$.
\end{itemize}
We observe that if $\lambda \in k \smallsetminus \zed$, then $M_\lambda\ang{1} \cong M_{\lambda+1}$, and $M_\lambda \not\cong M_\mu$ in $\rgr A$  if $\lambda \neq \mu$.  Thus the graded isomorphism classes of graded simples correspond to the shifts $X\ang{n}$ and $Y \ang{n}$, plus one module $M_\lambda$ for each element of $k \smallsetminus \zed$.
\end{proof}

Another indication of the comparative rigidity of $\rgr A$ is that there is only one outer automorphism of $A$ that preserves the graded structure.  This is the map $\omega: A \to A$  defined by $\omega(x) = y$ and $\omega(y) = -x$.  We will also denote the induced autoequivalences of $\rgr A$, $\rmod A$, and $A \lmod$ by $\omega$.  
We remark that the action of $\omega$ on the simples $X \ang{n}, Y \ang{n}$ is given by $\omega(X\ang{n}) \cong Y\ang{-n-1}$ and $\omega(Y \ang{n} ) \cong X \ang{-n-1}$.

For notational convenience,  we will also define $\omega$ on graded $k$-vector spaces as the functor that reverses grading: if $V$ is a graded vector space, we write $\omega V$ for the graded vector space given by $(\omega V)_n = V_{-n}$.   Then $\omega$  gives isomorphisms of graded vector spaces
\beq \label{omega-identity} 
\begin{split}
\Hom_A (\omega M, \omega M') \cong \omega \Hom_A(M, M') \\
\Ext_A(\omega M, \omega M') \cong \omega \Ext_A(M, M').	\end{split} \eeq
for any $M$ and $M'$ in $\rgr A$.  
If either $M$ or $M'$ is a graded $A$-bimodule, then it is an easy computation that the maps in~\ref{omega-identity} are $A$-module maps.

We now compute extensions between simple objects in $\rgr A$.  
We remind the reader that $\hom_A$ and $\ext_A = \ext^1_A$ refer to homomorphisms in the category $\rgr A$, and $\Hom_A$ and $\Ext_A$ refer to $\rmod A$.

\begin{lemma}\label{lem-Ext}
$(1)$ $\Ext_A(X, A) \cong (A/Ax)\ang{-1}$, and $\Ext_A(Y, A) \cong A/Ay$.

$(2)$ As a graded vector space, $\Ext_A(X,Y) = \Ext_A(Y, X) =k$, concentrated in degree 0.

$(3)$ 
$\Ext_A(X,X) = \Ext_A(Y,Y) = 0$.

$(4)$ If $\lambda \neq \mu$,  then $\ext_A(M_\lambda, M_\mu) = 0$, but $\ext_A(M_\lambda, M_\lambda) = k$.

$(5)$ Let $\lambda \in k \smallsetminus \zed$ and let $S \in \{X, Y\}$.  Then $\Ext_A(M_{\lambda}, S) = \Ext_A(S, M_{\lambda}) = 0$.
\end{lemma}

\begin{proof}
(1) We apply $\Hom_A(\blank, A)$ to the short exact sequence
\beq\label{s-star} \xymatrix{
0 \ar[r] & xA \ar[r] & A \ar[r] & X \ar[r] & 0 } \eeq
to obtain:
\[ \xymatrix{ 0 \ar[r] & A \ar[r] & A x^{-1} \ar[r] & \Ext_A(X, A) \ar[r] & 0. } \]
Thus we have computed that 
$\Ext_A(X, A) \cong A x^{-1} / A \cong (A / Ax) \ang{-1}$.
Applying $\omega$, we obtain from \eqref{omega-identity} that 
$\Ext_A(Y \ang{-1}, A) \cong  (A/Ay)\ang{1}$, and so 
$\Ext_A(Y, A)  \cong \Ext_A(Y \ang{-1}, A)\ang{-1} \cong A/ Ay$.

(2) Applying $\Hom_A(Y, \blank)$ to \eqref{s-star}, we obtain
\[0 \to \Ext_A(Y, A) \ang{1} \to \Ext_A(Y, A) \to \Ext_A(Y, X) \to  0. \]
Since by (1)  we know that $\Ext_A(Y, A)_j$ is 0 if $j \leq -1$ and $k$ if $j \geq 0$, we see that $\Ext_A(Y,X) \cong k$.  Applying $\omega$, we obtain that 
\[ k \cong \omega(k) \cong \Ext_A(\omega Y, \omega X) = \Ext_A(X \ang{-1}, Y \ang{-1}) \cong \Ext_A(X, Y).\]

(3)  By applying $\Hom_A(X, \blank)$ to \eqref{s-star}, we obtain a long exact sequence:
\[ 
0 \to \Hom_A(X, X) \to \Ext_A(X, A)\ang{1} \to \Ext_A(X, A) \to \Ext_A(X, X) \to 0 \]
and a similar computation shows that $\Ext_A(X, X) = 0$.  Applying $\omega$ again we see that $\Ext_A(Y, Y) = 0$.

(4) We apply $\hom_A(\blank, M_\mu)$ to the exact sequence
\[ \xymatrix{
0 \ar[r] & A \ar[r]^{z+\lambda} & A \ar[r] & M_\lambda \ar[r] & 0 }\]
to obtain
\[\xymatrix@=.4cm{
0 \ar[r] & \hom_A(M_\lambda, M_\mu) \ar[r] & k[z]/(z+\mu) \ar[rr]^{z + \lambda} &&
	k[z]/(z+\mu) \ar[r] & \ext_A(M_\lambda, M_\mu) \ar[r] & 0. }\]
It is immediate that $\ext_A(M_\lambda, M_\mu) \cong \hom_A(M_\lambda, M_\mu)$ and is 0 if $\lambda \neq \mu$ and $k$ if $\lambda = \mu$.

(5) This proof is similar to the others, and we omit it.
\end{proof}

We note one consequence of Lemma~\ref{lem-Ext}: since the  sequences
\beq\label{YX} \xymatrix{
0 \ar[r] & Y \ar[r] & A/zA \ar[r] & X \ar[r] & 0 } \eeq
and
\beq \label{XY} \xymatrix{
0 \ar[r] & X \ar[r] & (A/(z-1)A) \ang{1} \ar[r] & Y \ar[r] & 0 } \eeq
do not split, they and their shifts are the only nonsplit extensions in either $\rgr A$ or $\rmod A$ involving the shifts of $X$ and $Y$.  In fact, in general we have:
\begin{lemma}\label{lem-indec}
Let $Z$ be an indecomposable graded $A$-module of finite length.  Then $Z$ is determined up to isomorphism by its Jordan-H\"older quotients.
\end{lemma}
\begin{proof}
This follows by induction from Lemmas~\ref{lem-simple} and \ref{lem-Ext}.
\end{proof}

The left $k[z]$-action on the modules $(A/zA) \ang{n}$ and $(A/(z-1)A)\ang{n+1}$ is especially nice.  We record this as:
\begin{lemma}\label{lem-comp}
Let $n \in \zed$.    Then $(z+n) (X\ang{n}) = 0 = (z+n) (Y \ang{n})$, and as graded left $k[z]$-modules, we have
\[ (A / zA) \ang{n} \cong (A/(z-1) A) \ang{n+1}\cong \bigoplus_{j \in \zed} \frac{k[z]}{(z+n)}.\]
\end{lemma}
\begin{proof}
This follows from the exact sequences \eqref{XY} and \eqref{YX} and the computation
$(z+n) x^n = x^n z \in x^{n+1} A$ in the quotient ring of $A$.
\end{proof}

Since  $M_\lambda \cong \bigoplus_{j \in \zed} k[z]/(z+\lambda)$ as a left $k[z]$-module, we see that any graded $A$-module of finite length, when considered as a left $k[z]$-module, is supported at finitely many $k$-points of $\Spec k[z]$.   

\begin{defn}
If $M$ is a graded right $A$-module of finite length, we define  the {\em support  
 of $M$}, $\Supp M$, to
 be the   support of $M$ as a left $k[z]$-module.  We are particularly interested in the cases when $\Supp M \subset \zed$ (we say that $M$ is { \em integrally supported}) or when $\Supp M = \{ n\}$ for some $n \in \zed$ (we say that $M$ is {\em simply supported at $n$}).  Lemma~\ref{lem-simple} and Lemma~\ref{lem-comp} show that $X\ang{n}$ and $Y \ang{n}$ are the unique simples simply supported at $-n$.
 \end{defn}

We now  turn to understanding graded projective $A$-modules.  We will say that an injection $f: P \hookrightarrow Q$ of rank 1 graded projective modules is a {\em maximal embedding} if there is no $f': P \hookrightarrow Q$ with $\im (f') \supsetneq \im(f)$.   

\begin{lemma}\label{lem-maxemb}
Let $P$ and $Q$ be rank 1 graded projective $A$-modules, and let $f: P \to Q$ be a  maximal embedding.  Then the module 
 $Q/f(P)$ is semisimple and integrally supported.
\end{lemma}
\begin{proof} 
Let $N = Q/f(P)$.  As $Q_A$ is 1-critical, $N$   has finite length. By Lemma~\ref{lem-simple} $N$ has a finite composition series whose factors are isomorphic to $X\ang{n}, Y \ang{n}$, or $M_\lambda$, with $\lambda \not\in \zed$.

Suppose that some $M_\lambda$ is a subfactor of $N$.  Then we have $f(P) \subseteq Q_1 \subseteq Q_2 \subseteq Q$, with $Q_2/Q_1 \cong M_\lambda$, and so $Q_1 = (z+\lambda) Q_2$.  But now
$ (z+\lambda)^{-1} f(P) \subseteq (z+\lambda)^{-1} Q_1 = Q_2 \subseteq Q$, contradicting the maximality of $f$.  

Thus $N$ is integrally supported.
  If $N$ is not semisimple,  then by Lemma~\ref{lem-Ext}, $N$ has a subfactor isomorphic to either $(A / zA) \ang{n}$ or $(A/ (z-1) A) \ang{n+1}$.  In either case, arguing as above we have $(z+n)^{-1} f(P) \subseteq Q$, contradicting the maximality of $f$. 
 \end{proof}

  We have the following easy consequence of Lemma~\ref{lem-maxemb}:
\begin{lemma}\label{lem-asymptotic}
Let $P$ be  a rank 1 graded projective $A$-module.  Then for all $n \gg 0$, we have $\hom_A(P, X\ang{n}) = \hom_A(P, Y\ang{-n}) = k$.
\end{lemma}
\begin{proof}
Let 
\[ 0 \to P \to A \to M \to 0\]
be a maximal embedding.    Then there are exact sequences
\[ \Hom_A(M, X) \to \Hom_A(A, X) \to \Hom_A(P, X) \to \Ext_A(M, X)\]
and
\[ \Hom_A(M, Y) \to \Hom_A(A, Y) \to \Hom_A(P,Y) \to \Ext_A(M, Y).\]
By Lemma~\ref{lem-maxemb}, $M$ is semisimple and and integrally supported.  By Lemma~\ref{lem-Ext},  the vector spaces $\Hom_A(M, X)$, $\Hom_A(M, Y)$, $\Ext_A(M, X)$, and $\Ext_A(M,Y)$ are finite-dimensional; so for $n \gg  0$ we have $\hom_A(P, X \ang{n}) = \hom_A(A, X \ang{n}) = X_{-n} = k$ and $\hom_A(P, Y \ang{-n}) = \hom_A(A, Y \ang{-n}) = Y_n = k$.
\end{proof}

In fact, more is true:  the simple factors of a module $P$ partition the set of integrally supported simples.  
\begin{proposition}\label{prop-duality}
Let $P$ be a rank 1 graded projective $A$-module, and let $n \in \zed$.  Then exactly one of the following is true:

 $(1)$ We have
\[ \ext_A(Y\ang{n}, P) = \hom_A(P, X\ang{n}) = k\]
and
\[ \hom_A(P, Y\ang{n}) = \ext_A(X\ang{n}, P) = 0;  or \]

$ (2)$ we have
 \[\hom_A(P, X\ang{n}) = \ext_A(Y\ang{n}, P) = 0\]
 and
 \[\ext_A(X\ang{n}, P) = \hom_A(P, Y\ang{n}) = k.\]
\end{proposition}
\begin{proof}
We first prove the result for $n = 0$.  
By Lemma~\ref{lem-Ext}, we have $\ext_A(X, A) = 0 =  Y_0 = \hom_A(A,Y)$ and $\ext_A(Y, A) = k = X_0 = \hom_A(A, X)$.  Thus the claim holds for $A = P$.  For general $P$, let  
\beq\label{eq1} 0 \to P \to A \to M \to 0 \eeq
be a maximal embedding of $P$ in $A$. By Lemma~\ref{lem-maxemb},  $M$ is semisimple and  integrally supported.  Since $Y$ is not a factor of $A$, we know that $Y$ is not a factor of $M$, and so  $\hom_A(M, Y) = \hom_A(Y, M) = \ext_A(M, X) = \ext_A(X,M) = 0$.  Further, by Lemma~\ref{lem-Ext}, there are $k$-vector space isomorphisms
\beq\label{eq2}
\hom_A(M, X) \cong \ext_A(M,Y) \cong \ext_A(Y, M) \cong \hom_A(X,M).
\eeq

Via the long exact $\Ext$ sequence, \eqref{eq1} induces a diagram with exact rows

\[ \xymatrix{
& 0 \ar[r] & \hom_A(X, M) \ar[d]^{\cong} \ar[r] & \ext_A(X, P) \ar[r] & \ext_A(X, A) = 0 \\
0  \ar[r] & \hom_A(P, Y) \ar[r] & \ext_A(M,Y) \ar[r] & 0}\]
where the vertical isomorphism comes from \eqref{eq2}.  
Thus $\ext_A(X,P) \cong \hom_A(P,Y)$.  
Applying $\omega$, we see that $\ext_A(Y, P') \cong \hom_A(P', X)$, where $P' = (\omega P) \ang{1}$.  As $P$ ranges over all rank 1 graded  projectives, so does $P'$.  This shows that for all $P$ we have $\ext_A(Y, P) \cong \hom_A(P, X)$.  

To see that exactly one of (1) and (2) holds, observe that  \eqref{eq1} also induces a diagram with  exact rows
\[ \xymatrix{
 0 \ar[r] & \ext_A(Y, P) \ar[r] & \ext_A(Y,A) \ar[r] & \ext_A(Y, M) \ar[r] \ar[d]^{\cong} & 0 \\
&  0 \ar[r] & \hom_A(P, Y) \ar[r]& \ext_A(M,Y) \ar[r] & 0,}\]
  where again the vertical map comes from \eqref{eq2}.
  Thus  we see that
\[ k = \ext_A(Y, A) \cong \ext_A(Y,P) \oplus \hom_A(P, Y).\]
The claim follows,  and by shifting degrees we obtain the result for all $n$.
\end{proof}

 \begin{defn}\label{def-Fj}
  For any integer $j$ and rank 1 graded projective $P$, let $F_j(P)$ be the unique graded simple factor of $P$ supported at $-j$.  We will  say that two rank 1 projectives $P$ and $Q$ are  {\em $j$-congruent} if $F_j(P) \cong F_j(Q)$, and are {\em $j$-opposite} if they are not $j$-congruent.  \end{defn}
 
  Note  that $F_j(P)$ is well-defined by  Proposition~\ref{prop-duality}.

\begin{lemma}\label{lem-simplefactor}
If $P$ and $Q$ are rank 1 graded projective $A$-modules that  are $j$-congruent for all $j \in \zed$, then $P \cong Q$.
\end{lemma}
\begin{proof}
Let
\[ 0 \to P \to Q \to M \to 0 \]
be a maximal embedding  of $P$ in $Q$.  By Lemma~\ref{lem-maxemb},  $M$ is semisimple and integrally supported.   Let $N$ be a simple summand of $M$.  Consider the sequence
\[ 0 \to \hom_A(M, N) \to \hom_A(Q, N) \to \hom_A(P, N) \to \ext_A(M, N).\]
Now, by Proposition~\ref{prop-duality}, there is no $j$ such that $M$ contains both $X \ang{j}$ and $Y \ang{j}$ as composition factors, and $\dim_k \hom_A(Q, N) = 1$.  Then, by Lemma~\ref{lem-Ext}, $\ext_A(M, N) =0$, and so $ \hom_A(P, N) = 0$.  This implies that  for any $j \in \Supp M$, the modules $P$ and $Q$ are $j$-opposite.  Since $P$ and $Q$ are always $j$-congruent, we must have $M=0$ and $P \cong Q$.
\end{proof}

This lemma again shows the  comparative rigidity of the graded category:  it is certainly not true that the (ungraded) simple factors of a projective module determine it up to isomorphism, since every projective module is a generator for $\rmod A$.

We record for future reference the following routine consequence of the fact that $A$ is hereditary:

\begin{lemma} \label{lem-split}
Let $P$ be a finitely generated graded projective $A$-module.  Then $P$ splits completely as a direct sum of rank 1 graded projective modules. \qed
\end{lemma}

\section{The Picard group of $\rgr A$}\label{sec-Pic}
In this section we will calculate the Picard group of $\rgr A$. Let $D_\infty$ denote the infinite dihedral group, and let $(\zed/2\zed)^{(\zed)}$ be the direct sum of countably many copies of $\zed/2\zed$.  We will show that there is an  exact sequence of groups
\[ \xymatrix{
1 \ar[r] & (\zed/2\zed)^{(\zed)} \ar[r] & \Pic(\rgr A) \ar[r] & D_\infty \ar[r] & 1 }\]
and that in fact $\Pic(\rgr A)$ is isomorphic to the restricted wreath product \[(\zed /2\zed) \mywr_{\zed} D_\infty \cong (\zed/2\zed)^{(\zed)} \rtimes D_\infty.\]

We first investigate the automorphism group of the $\zed$-algebra $\bbar{A}$ and the map
\eqref{Aut-Pic} for $A$.  
We establish notation: define $m_{ij} \in \bbar{A}_{ij} = A_{j-i}$ to be the canonical $k[z]$-module generator of $A_{j-i}$; that is, $m_{ij}$ is $x^{j-i}$ if $j \geq i$ and $y^{i-j}$ if $i> j$.   

Recall that an automorphism of $\bbar{A}$ is {\em inner} if it is in the kernel of \eqref{Aut-Pic}.

\begin{lemma}\label{lem-autA} 
Let $\gamma$ be an automorphism of $\bbar{A}$ of degree 0.  Then $\gamma$ is  inner.
\end{lemma}

\begin{proof}
For all  $i, j \in \zed$ there is a unit $\zeta_{ij} \in k[z]$ such that $\gamma(m_{ij}) = \zeta_{ij} m_{ij}$.  Thus  $\zeta_{ij} \in k^*$, and in particular $\zeta_{ij}$ is central in $A$.   
Further, for all $n \in \zed$ we have $\zeta_{nn} = 1$.  Applying $\gamma$ to the identity
\[m_{n, n+1} m_{n+1, n} - m_{n, n-1} m_{n-1, n} = 1_{n},\]
we obtain that $\zeta_{n, n-1} \zeta_{n-1, n} = 1$ for all $n$, and so $\gamma(z \cdot 1_{n}) = z\cdot 1_{n}$.   This implies that for any $f \in k[z]$ and $i, j \in \zed$, we have
\[ \gamma(f \cdot m_{ij}) = \zeta_{ij} f \cdot m_{ij},\]
and in particular that $\zeta_{ij} \zeta_{j \ell} = \zeta_{i \ell}$ for all $i, j, \ell \in \zed$.   
 Thus if $v \in \bbar{A}_{ij}$, we have $\gamma(v) = \zeta_{ij} v = \zeta_{i0} \zeta_{0j} v =\zeta_{i0} v \zeta_{0j} $, and so $\gamma$ is inner by Theorem~\ref{thm-BR}.  
\end{proof}

\begin{corollary}\label{notwist-A}
If $\Phi: \rgr A \to \rgr S$ is a twist functor, then $S \cong A$ and $\Phi \cong \id_{\rgr A}$.
\end{corollary}
\begin{proof}
This follows from Lemma~\ref{lem-autA}  and Corollary~\ref{cor-notwist}.  \end{proof}

The next result is the most important application of the rigidity of the category $\rgr A$: the infinite dihedral group $D_\infty$ is a factor as well as a subgroup of $\Pic(\rgr A)$.  This gives us two integer invariants associated to each element of $\Pic(\rgr A)$. 

\begin{theorem}\label{thm-RST}
Let $\mathcal{F}$ be an autoequivalence of $\rgr A$.  Then there are unique integers $a = \pm 1$ and $b$ such that for all $n \in \zed$, we have
\[ \{ \mathcal{F}(X\ang{n}), \mathcal{F}(Y \ang{n})\} \cong \{ X \ang{an+b}, Y \ang{an+b} \} \]
and for all $\lambda \in k \smallsetminus \zed$,
\[ \mathcal{F}(M_\lambda) \cong M_{a \lambda + b}.\]
\end{theorem}
\begin{proof}
If $a$ and $b$ exist, they are clearly unique.

We  observe from Lemma~\ref{lem-simple} and Lemma~\ref{lem-Ext} that $\mathcal{F}$ must map the simple module $M_\lambda$ to some other simple $M_{\mu}$, since these are the only simples with nonsplit self-extensions.  Therefore, $\mathcal{F}$ must map integrally supported simples to integrally supported simples.  Further, by Lemma~\ref{lem-Ext}, for all $n \in \zed$ the pair  $\{X\ang{n}, Y\ang{n}\}$ must map to some other pair $\{X \ang{n'}, Y\ang{n'}\}$, as these pairs form the only nonsplit extensions of two nonisomorphic  simples.    In other words, there is a bijection $g: k \to k$ such that:
\begin{enumerate}
\item If $\lambda \in \zed$, then $g(\lambda) \in \zed$ and $\F(\{X\ang{\lambda}, Y\ang{\lambda}\}) \cong \{X\ang{g(\lambda)}, Y\ang{g(\lambda)} \}$.
\item If $\lambda \in k \smallsetminus \zed$, then $g(\lambda) \not\in \zed$ and $\F(M_\lambda) \cong M_{g(\lambda)}$.
\end{enumerate}

Now consider the functor $\F_0 = \F(\blank \otimes_{k[z]} A)_0: \rmod k[z] \to \rmod k[z]$.  We claim that for all $\lambda \in k$, we have   $\F_0(k[z]/(z+\lambda)) \cong k[z]/(z+g(\lambda))$.  This follows from Lemma~\ref{lem-comp} if $\lambda \in \zed$, and from the definition of $M_\lambda$ and $M_{g(\lambda)}$ if $\lambda \not\in\zed$.  
In particular, as $g$ is a bijection, $\F_0$ is invertible and so an autoequivalence of $\rmod k[z]$.  But the only such functors act via automorphisms of  $k[z]$, so there are constants $a,b \in k $ such that $g(n) = an+b$ for all $n$.  As $g$ maps $\zed$ bijectively to $\zed$, $a$ must be $\pm 1$ and $b$ must be an integer.
\end{proof}

\begin{defn}
If $\F$ is an autoequivalence of $\rgr A$, the integer $b$ above is the {\em rank} of $\F$, and the integer $a$ is the {\em sign} of $\F$.  We say that $\F$ is {\em odd} if it has sign -1, and is {\em even} if it has sign 1.  
We say that  $\F$ is {\em numerically trivial} if it is even and has rank 0. 
\end{defn}

 For example, $\shift_A^n$ is even of rank $n$, and $\omega$ is odd of rank -1.  
  
Since the set of maps $g: \zed \to \zed$ of the form $n \mapsto \pm n + b$ is isomorphic to $D_\infty$, and rank and sign clearly behave well with respect to composition of functors, from Theorem~\ref{thm-RST}  we immediately obtain: 

\begin{corollary}\label{cor-Piczquotient}
Let $\Picz(\rgr A)$ be the subgroup of $\Pic(\rgr A)$ of numerically trivial autoequivalences.    
Then $\Pic(\rgr A) \cong \Picz(\rgr A) \rtimes D_\infty$.
\end{corollary}

\begin{proof}
This follows from Theorem~\ref{thm-RST} and the fact that the subgroup
\[ \ang{\omega, \shift_A} \subseteq \Pic(\rgr A) \]
is isomorphic to $D_\infty$.
\end{proof}

As a second corollary, we see that to check if two autoequivalences of $\rgr A$ are isomorphic, it suffices to see if they agree on integrally supported simples.

\begin{corollary}\label{cor-intsimple}
Let $\F$ and $\F'$ be autoequivalances of $\rgr A$.  Then $\F \cong \F'$ if and only if $\F(S) \cong \F'(S)$ for all integrally supported simples $S$.
\end{corollary}
\begin{proof}
Suppose that $\F(S) \cong \F'(S)$ for all integrally supported simples $S$.  Then by Theorem~\ref{thm-RST}, $\F$ and $\F'$ have the same rank and sign, and so $\F(S) \cong \F'(S)$ for any simple module $S$. Let $N$ be a finitely generated graded $A$-module; we write  $N = P \oplus Z$ where $P$ is torsion-free (and hence projective)  and $Z$ is torsion.   
Let $Z = Z_1 \oplus Z_2 \oplus \cdots \oplus Z_n$, where the $Z_i$ are indecomposable.  By Lemma~\ref{lem-indec}, $\F(Z_i) \cong \F'(Z_i)$ for all $i$, so $\F(Z) \cong \F'(Z)$.  

   By Lemma~\ref{lem-split}, $P$ splits completely as a direct sum of rank 1 graded projectives; we  write $P \cong P_1 \oplus \cdots \oplus P_n$, where the $P_i$ have rank 1.  Since $\F$ and $\F'$ agree on simples,  for each $i$ the modules  $\F P_i$ and $\F' P_i$ have isomorphic simple factors and so by Lemma~\ref{lem-simplefactor}, $\F P_i \cong \F' P_i$ for all $i$.  Thus $\F P \cong \F' P$, and so $\F N \cong \F P \oplus \F Z \cong \F' P \oplus \F' Z \cong \F' N$.  Thus $\F$ and $\F'$ are weakly isomorphic; by Lemma~\ref{lem-autA} and Corollary~\ref{cor-noweak}, $\F \cong \F'$.

The other direction is trivial.
\end{proof}

A priori, it is not clear that $\Picz(\rgr A)$ is nontrivial.  
It turns out, however, that this group is  quite large.  We exhibit generators for it in the next proposition.

We recall the notation from Definition~\ref{def-Fj} that if $P$ is  a rank 1 graded projective module, then $F_j(P)$ is the unique simple factor of $P$ supported at $-j$.

\begin{proposition}\label{prop-invsexist}
For any $j \in \zed$ there is a numerically trivial  automorphism $\inv_j$ of $\rgr A$  such that $\inv_j(X \ang{j}) \cong Y \ang{j}$, $\inv_j(Y \ang{j}) \cong X \ang{j}$, and if $j' \neq j$, then 
$\inv_j(X \ang{j'}) \cong X \ang{j'}$ and $\inv_j (Y \ang{j'}) \cong Y \ang{j'}$.  Further, for any $i, j \in \zed$, 
\beq\label{eq-obvious}
\shift^i_A \inv_j \cong \inv_{i+j} \shift^i_A,
\eeq
and $\inv_j^2 \cong \id_{\rgr A}$.  
\end{proposition}
\begin{proof}
Suppose that $\inv_0$ exists as described.  Then the automorphism $\inv_j = \shift_A^{j} \inv_0 \shift_A^{-j}$ has the required properties, and \eqref{eq-obvious} is satisfied by construction.  Thus it suffices to construct $\inv_0$.

We first define the action of $\inv_0$ on rank 1 graded projective modules.  Let $P$ be such a module.  By  Proposition~\ref{prop-duality}, we know that $\hom_A(P, X \oplus Y) = k$; thus there is a unique submodule $N$ of $P$ such that the sequence
\[ \xymatrix{ 0 \ar[r] & N \ar[r] & P \ar[r] & X \oplus Y }\]
is exact.   Formally, $N$ is the reject of $X \oplus Y $ in $P$; we remark that $N$ is 0-opposite to $P$.   Let $\inv_0(P) = N$.

Suppose that $Q$ is a rank 1 projective that is 0-opposite to $P$.  Let $f: Q \to P$ be a nonzero map.  We claim that $f(Q) \subseteq N$.  
To see this, let $S = F_0(P)$.  Since $Q$ is 0-opposite to $P$, we have that $\hom_A(Q, S) = 0$.  Thus  if we let $M = \coker f$, 
from the exact sequence
\[ 0 \to \hom_A(M, S) \to \hom_A(P, S)  \to \hom_A(Q, S) = 0\]
we see that $\hom_A(M, S) \cong \hom_A(P, S) =  k$.  Thus there is some $N'$ with $f(Q) \subseteq N' \subseteq P$ and  $P / N' \cong S$; but now since $S \cong P/N$ and $\hom_A(P, S) \cong k$, we have $N'=N$ and so $f(Q) \subseteq N$, as claimed.

We next define the action of $\inv_0$ on morphisms between rank 1 projectives.  Let  $f: P \to P'$, where $P$ and $P'$ are rank 1 graded projective modules.  Define  $\inv_0(f): \inv_0(P) \to \inv_0(P')$ to be $f |_{\inv_0(P)}$.  We verify that this is well-defined --- that is, that $f(\inv_0(P)) \subseteq \inv_0(P')$.  If $P'$ and $P$ are 0-congruent, then $\inv_0(P) $ and $P'$ are 0-opposite, so by the claim above we have $f(\inv_0(P)) \subseteq \inv_0(P')$.  On the other hand, if $P'$ and $P$ are 0-opposite, then $f(P) \subseteq \inv_0(P')$ so certainly  $f(\inv_0(P)) \subseteq \inv_0(P')$.

It is clear that $\inv_0$ behaves functorially on maps between rank 1 projectives; that is, it sends identities to identities and preserves commuting triangles.  By standard arguments, $\inv_0$ extends uniquely to a functor defined on all modules and morphisms in $\rgr A$. 

We show that $\inv_0$ has the properties claimed.  Let $P$ be a rank 1 graded projective module.  By Lemma~\ref{lem-comp} we have $\inv_0^2 P = z P$.  Without loss of generality, we may regard $P$ as a submodule of the graded quotient ring of $A$; thus if we define $\inv_0^{-1}(P) = z^{-1} \inv_0(P)$ and extend to all of $\rgr A$ it is clear that $\inv_0 \inv_0^{-1} = \inv_0^{-1} \inv_0 = \id_{\rgr A}$.  Thus $\inv_0$ is an automorphism of $\rgr A$.  We also note that if $\inv_0$ behaves as claimed on simple modules, then for any integrally supported simple $S$, we have $\inv_0^2(S) \cong S$, and so by Corollary~\ref{cor-intsimple}, $\inv_0^2 \cong \id_{\rgr A}$.  

Thus it remains to establish that $\inv_0$ acts as claimed on simples.  First, since $\inv_0 A  = xA$, by applying $\inv_0$ to the exact sequence
\[ 0 \to xA \to A \to X \to 0\]
we obtain
\[ 0 \to zA \to xA \to \inv_0(X) \to 0\]
and so $\inv_0(X) \cong xA / zA \cong Y$.  

By Theorem~\ref{thm-RST}, we see that $\inv_0$ has rank 0 and sends $Y$ to $X$.  
We show that $\inv_0$ is even by computing $\inv_0(X \ang{1})$.  Let $P = (z+1) A + x^2 A$.   It is straightforward to see that $A/P \cong X \ang{1}$, and that $P/x^2 A \cong X$.  Thus $x^2A = \inv_0(P)$, and applying $\inv_0$ to the exact sequence
\[ 0 \to P \to A \to X \ang{1} \to 0 \]
gives
\[ 0 \to x^2 A \to x A \to \inv_0(X \ang{1}) \to 0.\]
Thus $\inv_0(X\ang{1}) \cong xA /x^2 A \cong  X \ang{1}$.  Therefore the image of $\inv_0$ in the factor group $D_\infty$ of $\Pic(\rgr A)$ acts on $\zed$ by sending 0 to 0 and 1 to 1, and so must be the identity.

The simple factors of $\inv_0(A) \cong A\ang{1}$ are $\{X\ang{j}\}_{j \geq 1}$ and $\{Y \ang{j}\}_{j \leq 0}$.  Thus for all  $j \neq 0$ we have $F_j(A) \cong F_j(\inv_0(A)),$ which is isomorphic to $\inv_0(F_j(A))$ since $\inv_0$ is numerically trivial.  Theorem~\ref{thm-RST} shows that $\inv_0$ has the desired behavior on integrally supported simples.  
\end{proof}

It is clear that $\inv_i \inv_j = \inv_j \inv_i$ and that the subgroup of $\Pic(\rgr A)$ generated by the $\{\inv_j\}$ is isomorphic to the  countable direct sum $(\zed/2\zed)^{(\zed)}$.  It is convenient to identify this with the set of finite subsets of $\zed$, which we denote $\zedfin$.  The group operation on $\zedfin$ is exclusive or, which we write as $\oplus$; thus   for $K, J \in \zedfin$, we have 
\[ K \oplus J = (K \cup J) \smallsetminus (K \cap J).\]
The identity element of $\zedfin$ is $\emptyset$.   

There are multiplicative and additive actions of $\zed$ on $\zedfin$, given by 
\[ nJ = \{ nj \st j \in J\} \]
and
\[ J + n = \{ j+n \st j \in J \}\]
for any $n \in \zed$ and $J \in \zedfin$.  In particular, $D_\infty$ is naturally a subgroup of $\Aut(\zedfin)$.

To each $J \in \zedfin$, we may associate an automorphism of $\rgr A$, which we write either $\inv_J$ or $\invbrak{J}$.  It is defined by
\[ \invbrak{J} = \inv_J = \prod_{j \in J} \inv_j,\]
with inverse
\[ \inv_J^{-1} = \Bigl( \prod_{j \in J} (z+j)^{-1} \Bigr) \inv_J.\] 
For completeness, we define $\inv_{\emptyset} = \id_{\rgr A}$.
Because $\inv_J^2 \cong \id_{\rgr A}$, we refer to the functors $\inv_J$ as {\em involutions.}  We note that if $P$ is a rank 1 graded projective $A$-module, then $\inv_J P$ is the unique submodule of $P$ that is maximal with respect to the property that it is $j$-opposite to $P$ for all $j \in J$.  (In fact, it is possible to give a categorical definition of the functors $\inv_J$ as the autoequivalences of $\rgr A$ that are subfunctors of the identity on projectives and whose square is naturally isomorphic to $\id_{\rgr A}$.  We will not prove this assertion.) 

We now show that the subgroup of involutions of $\rgr A$ is in fact equal to $\Picz(\rgr A)$.  

\begin{corollary}\label{cor-Picz}
Define
\begin{align*} \Lambda: \zedfin & \to \Picz(\rgr A)\\
	J & \mapsto \inv_J
	\end{align*}
Then $\Lambda$ is a group isomorphism.
\end{corollary}
\begin{proof}
Since $(\inv_j)^2 \cong \id_{\rgr A}$ for all $j$, we have that $\inv_K \inv_J \cong \invbrak{K\oplus J}$.  Therefore $\Lambda$ is a group homomorphism, and it is clearly injective.  
We prove surjectivity.  

Suppose that $[\mathcal{F}] \in \Picz(\rgr A)$, and let $P = \mathcal{F}(A)$.
Since $\F$ is numerically trivial, $ \F(F_j(A)) \cong  F_j(P) $ for all $j$; thus by Lemma~\ref{lem-asymptotic}, the set 
\[ J = \{ n \st \F(X\ang{n}) \cong Y \ang{n} \}\] 
is finite.  Since for all integrally supported simples $S$ we have $\inv_J(S) \cong \F(S)$,
 by Corollary~\ref{cor-intsimple} it follows that $\F \cong \inv_J$. 
\end{proof}

\begin{corollary}\label{cor-new}
The automorphism group of  $\bbar{A}$ is generated by inner automorphisms, by the canonical principal automorphism $\alpha$, and by the map $\bbar{\omega}: \bbar{A} \to \bbar{A}$ given by 
$\bbar{\omega}(m_{ij}) = m_{-i,-j}$ if $j \geq i$ and $\bbar{\omega}(m_{ij}) = -m_{-i, -j}$ if $j < i$.    In particular, if $\pi$ is the map defined in \eqref{Aut-Pic}, then 
\[ \im(\pi) = \ang{\omega, \shift_A } \subseteq \Pic(\rgr A).\]
\end{corollary}
\begin{proof}
As $\bbar{\omega}$ is clearly induced from the ring automorphism $\omega: A \to A$, it is well-defined and is an automorphism of $\bbar{A}$.  We note that $\pi(\bbar{\omega}) \cong \omega$ and $\pi(\alpha) \cong \shift_A$.

We consider the map \eqref{Aut-Pic}.  Let $\beta \in \Aut(\bbar{A})$.  By Theorem~\ref{thm-RST}, there is some $\gamma \in \ang{\alpha, \bbar{\omega}}$ so that $\pi(\beta \gamma^{-1})$ is a numerically trivial autoequivalence of $\rgr A$.  Thus without loss of generality we may suppose that $\pi(\beta) \in \Pic_0(\rgr A)$.    We will show that this implies that $\beta$ is inner.

Note that if $M$ is a cyclic graded right $A$-module, then $M_{\beta}$ is also cyclic as an $\bbar{A}$-module and therefore as an $A$-module.  Now, the rank 1 cyclic graded projective right $A$-modules are precisely the shifts $A \ang{n}$; these are the rank 1 graded projectives $P$ so that
\[ \{ i \st F_i(P) \cong X\ang{i} \} = [n, \infty) \]
for some $n \in \zed$.  

By Corollary~\ref{cor-Picz}, there is some $J \in \zedfin$ so that $\pi(\beta) \cong \inv_J$.  Suppose that $J \neq \emptyset$, and let 
\[ j = \min J - 1.\]
Let $P = \inv_J(A \ang{j}) \cong (A\ang{j})_{\beta}$.   Then
\[ \{i \st F_i(P) \cong X \ang{i} \} = \{j\} \sqcup \bigl( [j+2, \infty) \smallsetminus J \bigr).\]
Therefore, $P$ is not cyclic, a contradiction.

Thus $J = \emptyset$ and so $\pi(\beta) \cong \id_{\rgr A}$; thus by definition $\beta$ is inner.
\end{proof}
%
%

We note the contrast between the the graded and ungraded behavior of $A$.  Let us write the ungraded Picard group of $A$ as $\Pic(\rmod A)$.  Then a  theorem of Stafford \cite[Corollary~4.5(i)]{St2} says that the natural map from $\Aut(A) \to \Pic(\rmod A)$ is surjective.  On the other hand, Corollary~\ref{cor-new} shows that 
\[ \Pic(\rgr A) = \Picz(\rgr A) \rtimes \im \pi,\]
and we have seen that $\Picz(\rgr A)$ is large --- in particular, it is not finitely generated.

\begin{corollary}\label{cor-Picr}
The group $\Pic({\rgr A})$ is generated by $\shift_A$, $\omega$, and $\inv_0$, and is isomorphic to the restricted wreath product $(\zed/2\zed) \mywr_{\zed} D_\infty \cong \zedfin \rtimes D_\infty$.
\end{corollary}
\begin{proof}
This follows from Corollary~\ref{cor-Piczquotient}, Corollary~\ref{cor-Picz}, \eqref{eq-obvious}, and the computation $\omega \inv_j \omega = \inv_{-1-j}$.
\end{proof}

\begin{corollary}\label{cor-eRST}
Let $\F$ be an autoequivalence of $\rgr A$.  Then there are a unique integer $b$ and a unique $J \in \zedfin$ such that if $\F$ is even, then $\F \cong \shift^b_A \circ \inv_J$, and if  $\F$ is odd, then $\F \cong \shift^{b}_A \circ  \inv_J \circ \omega$.
 \qed
\end{corollary}

We end the section with several technical lemmas.  First, it is useful to have some explicit formulae for computing involutions.  

\begin{lemma}\label{lem-form}
$(1)$ For any $J \in \zedfin$, 
\[\inv_J A = \bigcap_{i \in J} \inv_i A.\]

$(2)$ For any $ i \in \zed$, consider $y^i A \subseteq A[y^{-1}]$.  Then for any $j \in \zed$, we have 
\[ \inv_j y^i A = y^i \inv_{j+i} A.\]

$(3)$ For any $i \in \zed$, 
\[\inv_i A = \brackarr{ 
	(z+i)A + x^{i+1}A	& \mbox{ if $i \geq 1$} \\
	xA				& \mbox{ if $i = 0$} \\
	yA				& \mbox{ if $i = -1$} \\
	(z+i)A + y^{-i}A	& \mbox{ if $i \leq -2$}
} \]
\end{lemma}
\begin{proof}
(1) We note that $\inv_J A$ is  $i$-opposite to $A$ for all $i \in J$; therefore, $\inv_J A \subseteq \inv_i A$ for any $i \in J$.  For any $i \in \zed$ and rank 1 graded projective $P$, we know that  $\inv_i P$ is a maximal submodule of $P$.  Therefore, $\inv_J A$ has height $\#J$ in $A$; as the length of $A / (\bigcap_{i \in J} \inv_i A)$ is $\#J$, we conclude that $\inv_J A = \bigcap_{i \in J} \inv_i A$.

(2) As graded modules, $\inv_j y^i A \cong \inv_j \shift_A^{-i} A \cong \shift_A^{-i} \inv_{j+i}A \cong y^i \inv_{j+i} A$.  Furthermore, both are maximal submodules of $y^i A$.  Therefore they must be equal.

We leave the computation of (3) to the reader.
\end{proof}

We also give two  lemmas that we will use to understand when an  autoequivalence of $\rgr A$ is generative.  

\begin{lemma}\label{lem-generative}
If $\F$ is an autoequivalence of $\rgr A$, then $\F$ is $P$-generative if and only if the set $\{ \F^n P \}$ generates all integrally supported graded simple $A$-modules.
\end{lemma}
\begin{proof}
One direction is clear.  For the other, suppose that $\{ \F^n P \}$ generates all integrally supported graded simples.  It suffices to prove that $\{ \F^n P \}$ generates all rank 1 graded projectives. Let $Q$ be an arbitrary rank 1 graded projective, and let $\psi: P \hookrightarrow Q$ be a maximal embedding.  By Lemma~\ref{lem-maxemb}, $Q/\psi(P)$ is semisimple and integrally supported, and thus generated by $\{ \F^n P \}$.  Thus there is a surjection $\phi: \bigoplus_{i=1}^j \F^{n_i} P \twoheadrightarrow Q/\psi(P)$ that by projectivity of the $\F^{n_i}P$ lifts to a map $\bbar{\phi}:  \bigoplus_{i=1}^j \F^{n_i} P \to Q$.  Since $\im \psi + \im \bbar{\phi} = Q$, therefore $\{\F^n P\}$ generates $Q$.
\end{proof}

\begin{lemma}\label{p-positive}
If $\F$ is a generative autoequivalence of $\rgr A$, then   $\F$ is even and has  nonzero rank.
\end{lemma}

\begin{proof}
  We note that $\omega \shift_A \cong \shift_A^{-1} \omega$.
Let $\F \in \Pic({\rgr R})$ be generative.  
Suppose that $\F$ is odd, so by Corollary~\ref{cor-eRST}, $\F \cong \shift_A^n \inv_J \omega$ for some $n \in \zed$ and $ J \in \zedfin$. 
Then for $K = (J+n) \oplus (-J - 1)$, we have, 
\[
\F^2 \cong \shift_A^n \inv_J \omega \shift_A^n \inv_J \omega \cong \shift_A^n \inv_J \shift_A^{-n} \invbrak{-J-1} \cong \inv_K.\]
Thus, $\F^4 \cong \id_{\rgr A}$, contradicting the generativeness of $\F$.  

Thus $\F$ is even; since the involutions $\inv_J$ are not generative, the rank of $\F$ must be nonzero.
\end{proof}

\section{Classifying rings graded equivalent to the Weyl algebra}\label{sec-main}

We are now ready to classify rings graded equivalent to $A$ and prove Theorem~\ref{ithm1}.     Theorem~\ref{thm-grMor-pullback} and the analysis of $\Pic(\rgr A)$ completed in the previous section are our basic tools; using them, we are able to compute twisted endomorphism rings extremely explicitly, and we obtain ``canonical'' representatives for each graded Morita equivalence class.    

As a first illustration of our approach, note that for any $n \geq 1$, $\shift_A^n$ is $A$-generative.   This can be seen either from Lemma~\ref{lem-generative}, or directly:  since $x^{n-1}A + yA = A$, therefore $A\ang{n-1}$ and $A \ang{-1}$ generate $A$ and (by induction) all $A\ang{-1}, \cdots, A\ang{n-1}$.  It is straightforward to see that if $\F =\shift_A^n$, then  $\End^{\F}_A(A) \cong A\ver{n} = \bigoplus_{i \in \zed} A_{ni}$, the $n$'th Veronese of $A$.   Thus by Proposition~\ref{prop-pullback}, $A$ is graded equivalent to $A\ver{n}$, which is of course the $\zed/n\zed$-invariant ring of $A$.   
This is in marked contrast to most commutative graded rings, where if $n \geq 2$, we expect that a ring and its $n$'th Veronese will have the same $\Proj$, but not the same graded module category.  Note that   if  we grade $R = k [x,y]$ by $\deg x = 1$ and $\deg y = -1$, then $\shift^n_R$ is generative exactly when $n = \pm 1$; this tells us that no Veronese of $R$ is graded equivalent to $R$.

Our goal is to  compute (up to graded Morita equivalence) all rings that are graded equivalent to $\rgr A$.  In general we will see a combination of idealizers, as in the ring $B$ from the Introduction, and Veronese rings, as discussed above.  More specifically, let $f(z)$ be a monic polynomial in $z$, and let $n$ be a positive integer.  We define the  {\em generalized Weyl algebra defined by $f$ and $n$}, denoted
$\GWA(f,n)$, to be the $k$-algebra generated by $X$, $Y$, and $z$ subject to the relations:
\begin{align*}
Xz - zX & = n X	& Yz - zY & = -nY \\
XY & = f		& YX &  = f(z-n).
\end{align*}
These rings were defined by Bavula \cite{Bav}, and have been extensively studied by him.   In particular, Bavula proves the following:
\begin{theorem}\label{thm-Bav}
{\em (\cite[Proposition~1.3, Corollary~3.2, Theorem~5]{Bav})}
If $f$ does not have multiple roots or two distinct roots differing by a multiple of $n$, then $\GWA(f,n)$ is a simple hereditary Noetherian domain. \qed
\end{theorem}
We caution the reader that Bavula uses slightly different notation; in particular, for Bavula the parameter $n$ is always equal to 1.   Since $W(f,n) \cong W(g, 1)$ for an appropriate choice of $g$, this is purely notational.
  
For any $J \in \zedfin$, define $f_J =\prod_{j \in J} (z + j)$.    Now let $n$ be a positive integer, and let $J \subseteq \{ 0, \ldots, n-1\}$.  Define $\bbar{J} = \{0, \ldots, n-1\} \smallsetminus J$, and let $W = W(f_{\bbar{J}},n)$.  We define a ring $S(J, n)$ by 
\[ S(J, n) = k[z] + f_J W \subseteq W.\]
Note that $S(J,n)  = \I_W(f_J W)$, the idealizer of $f_J W$ in $W$.   As a subring of $W$, $S(J, n)$ is a domain; by  \cite[Theorem~4.3]{Rob}, it is hereditary and Noetherian.

\begin{example}\label{ex-rings}
(1) Let $n =1$.  There are two subsets of $\{0\}$.  If $J = \{0\}$, then $f_{\bbar{J}} = 1$ and so $W = W(1,1) \cong A[y^{-1}]$. Let $T= A[y^{-1}]$.  Then   $S(\{0\}, 1) \cong \I_{T}(zT) = \I_T(xT)$.   On the other hand, if $J =\emptyset$, then  $S(\emptyset,1) \cong \GWA(z, 1) \cong A$.

(2) Now let $n > 1$ and let $J = \emptyset$.  
Then $f_{{J}} = 1$, and  $S(\emptyset, n) = W(\prod_{i=0}^{n-1} (z+i),n)$,
 which is isomorphic to the $n$'th Veronese  $A\ver{n}$ via 
\begin{align*}
X & \mapsto x^n \\
Y & \mapsto y^n \\
z & \mapsto z.
\end{align*}
\end{example}

We will see that the rings $S(J,n)$ give all the graded  Morita equivalence classes of rings graded equivalent to $A$.  However, they are not quite unique representatives for these equivalence classes.  We call a pair $(J, n) \in \zedfin \times \zed$ an {\em admissible pair} if $n \geq 1$ and $J \subseteq \{ 0, \ldots, n-1\}$.  We will 
say that two admissible pairs $(J, n)$ and $(J', n')$ have the same {\em necklace type} if $n = n'$ and there is some integer $j$ such that $J + j \equiv J' \pmod n$.  The idea is that $J$ and $J'$ both encode a string of $n$ beads, some black and some white, and that two pairs have the same necklace type if the strings make identical necklaces once we join the ends; we are allowed to rotate necklaces but not to turn them over.  These necklaces are well-known combinatorial objects.  The number  of necklaces of $n$ black and white beads is 
\[ \frac{1}{n} \sum_{d | n} \phi(d) 2^{n/d},\]
where $\phi$ is the Euler totient function \cite[Exercise~7.112]{EC}.

We will prove:

\begin{theorem}\label{thm-class}
If $S$ is a $\zed$-graded ring, then  $\rgr A \simeq \rgr S$ if and only if there is an admissible pair  $(J,n)$ such that $S$ is graded Morita equivalent to $S(J, n)$.    Furthermore, $S(J,n)$ and $S(L, r)$ are graded Morita equivalent if and only if $(J,n)$ and $(L,r)$ have the same necklace type.
\end{theorem}
This proves Theorem~\ref{ithm1} from the Introduction.  We also obtain Theorem~\ref{ithm-neck}:   

\begin{corollary}\label{cor-necklacedata}
The graded Morita equivalence classes of rings graded equivalent to $A$ are in 1-1 correspondence with pairs $(\mathbb{J}, n)$ where $n$ is a positive integer and $\mathbb{J}$ is a necklace of $n$ black and white beads. \qed
\end{corollary}

Our program for proving Theorem~\ref{thm-class} has two parts.  First we will  understand conjugacy classes in $\Pic(\rgr A)$ and thus graded Morita equivalence classes of rings graded equivalent to $A$.  Then we will  compute representative rings for each equivalence class.  

 For the first part, we will need several easy combinatorial lemmas.  We begin by establishing notation.  For any positive integer $n$, define an operator $\partial_n$ on $\zedfin$ by setting  
 \[ \partial_n J = J \oplus (J-n).\]
   Note that $\partial_n(I\oplus J) = \partial_n I \oplus \partial_n J$.  
   
      Given $i, n \in \zed$ with $0 \leq i \leq n-1$ and $J \in \zedfin$, define 
 \[ J^n_i = \{ j \in \zed \st nj+i \in J\}.\]

\begin{lemma}\label{lem-comb}
For all $n$, the  operator $\partial_n$ is a one-to-one map from $\zedfin$ onto 
\[\{J \in \zedfin \st \#(J^n_i) \mbox{ is even for } 0 \leq i \leq n-1  \}.\]
  \end{lemma}
\begin{proof}
We first prove the lemma for $n=1$.  
Let $J \in \zedfin$ with $\#J = 2m$ and write $J = \{a_1 < b_1 < \cdots < a_m < b_m\}$.  Let 
\[K = \bigsqcup_{i=1}^m \{a_i + 1, a_i +2, \ldots, b_i - 1, b_i\}.\]
Then $K \cup (K-1) = \bigsqcup_{i=1}^m \{a_i, a_i + 1, \ldots b_i\}$, and $K \cap (K-1) = \bigsqcup_{i=1}^m \{a_i + 1, \ldots b_i -1\}$.  We see that $\partial_1 K =  K \oplus (K-1) = \bigsqcup_{i=1}^m \{a_i,b_i\} = J$.  (Thus $ \partial_1 K$ is in some sense the ``boundary'' of $K$.)

Now suppose that $K \neq \emptyset$.  Let $r$ be the maximal element of $K$.  Then $r \not\in K - 1$, so $r \in (K \cup (K-1)) \smallsetminus (K \cap (K-1)) = \partial_1 K$, and $\partial_1 K \neq \emptyset$.

Finally,  given $I \in \zedfin$, put $I' = I \cap (I-1)$.  Then we have 
\[\partial_1 I = \bigl(I \smallsetminus I' \bigr) \cup \bigl((I-1) \smallsetminus  I')\bigr).\]
  This is a disjoint union, and as $I$ and $I-1$ have the same number of elements, we have 
  $\#\partial_1 I = 2 \cdot \#(I \smallsetminus I')$.
   
   Now let $n>1$, and let $J \in \zedfin$.  Then 
   \[ \partial_n J = \bigsqcup_{i=0}^{n-1} n (\partial_1 J^n_i) + i.\]
   Thus the result for general $n$ follows from the result for $n=1$.
\end{proof}

\begin{lemma}\label{lem-comb2} 
$(1)$  Let $\F$ be a generative autoequivalence of $\rgr A$.  Then $\F$ is conjugate in $\Pic(\rgr A)$ to  some $\shift^n_A \inv_J$, where $(J,n)$ is an admissible pair.  

$(2)$ If $(J, n)$ and $(L, m)$ are admissible pairs,  then $\shift^{m}_A \inv_{L}$ and $\shift^n_A \inv_J$ are conjugate in $\Pic(\rgr A)$ if and only if $(J,n)$ and $(L,m)$ have the same necklace type.
\end{lemma}
\begin{proof}
(1)   By  Lemma~\ref{p-positive} $\F$ is even and has nonzero rank; thus by Corollary~\ref{cor-eRST}  $\F \cong \shift_A^n \inv_K$ for some $n \neq 0$ and some $K \in \zedfin$.  Since $\omega \F \omega$ has rank $-n$, without loss of generality we may assume that $n > 0$.  

Define the set $J \subseteq \{0, \ldots, n-1\}$ to be:
\[J = \{  i  \st \# K_i^n \mbox{ is odd }\}.\]
   Then for all $0 \leq i \leq n-1$ we have that $\#(J\oplus K)^n_i$ is even, so by Lemma~\ref{lem-comb} there is some $I \in \zedfin$ such that $\partial_n I = J \oplus K$.   That is, 
\[ \inv_I \shift^n_A \inv_K \inv_I \cong \shift^n_A \invbrak{(I-n) \oplus K \oplus I} = \shift^n_A \invbrak{\partial_n I \oplus K} = \shift^n_A \inv_J.\]

(2) Let $(J, n)$ and $(L, m)$ be admissible pairs.  Note that conjugating by $\omega$ sends a rank $n$ autoequivalence to a rank $-n$ autoequivalence, while conjugating by an even autoequivalence preserves rank.  Since $m, n \geq 1$, if $\shift^{m}_A \inv_{L}$ and $\shift^n_A \inv_J$  are conjugate in $\Pic(\rgr A)$ then  they are conjugate in the subgroup of even autoequivalences, and we have $m=n$.

Suppose that $m=n$.  Then by the above and by Corollary~\ref{cor-eRST}, $\shift^n_A \inv_L$ and $\shift^n   \inv_J$ are conjugate if and only if there are some $j \in \zed$ and $I \in \zedfin$ such that 
$\shift^j_A \inv_I \shift^n \inv_J \inv_I \shift^{-j}_A \cong \shift^n \inv_L$; this is equivalent to the existence of $I$ and $j$ such that 
\[ J = \partial_n I \oplus (L-j).\]
By Lemma~\ref{lem-comb},  such an $I$ exists if and only if $(L-j)^n_i$ and $J^n_i$ have the same parity for all $0 \leq i \leq n-1$.  But since $J$ and $L$ are both contained in $\{0, \ldots, n-1\}$ this is true if and only if $(L-j) \equiv J \pmod n$.  This is equivalent to  $(L,n)$ and $(J,n)$ having the same necklace type.
\end{proof}

\begin{lemma}\label{l-gen}
Let $\F = \shift_A^n \circ \inv_I$, where $n > 0$ and $I \subseteq \{0, \ldots, n-1\}$.  Then $\F$ is $A$-generative.
\end{lemma}
\begin{proof}
We introduce notation:  for $j \neq 0$ define sets $\Gamma_{j}$ and  $\Delta_j$ by
\[
\Gamma_{j} = \brackarr{ 
	(I+n) \oplus (I+2n) \oplus \cdots \oplus (I+jn)	& \mbox{ if $j>0$}\\
	I \oplus (I-n) \oplus \cdots \oplus (I+(j+1)n)	& \mbox{ if $j<0$}
	} \]
\[
\Delta_{j} = \brackarr{ 
	\{0, \ldots, nj-1\} 	& \mbox{ if $j>0$}\\
	\{-1, \ldots, nj\}		& \mbox{ if $j<0$.}
	} \]
(Thus $\invbrak{\Delta_j} A \cong A \ang{nj}$.)  Then we see that if $j \neq 0$,
\[
 \F^j A  \cong \invbrak{\Gamma_j} \shift_A^{nj} A \cong \invbrak{\Gamma_j} \invbrak{\Delta_j} A = \invbrak{\Gamma_j \oplus \Delta_j} A. \]

Let $J_j = \Gamma_j \oplus \Delta_j$ for $j \neq 0$.    By Lemma~\ref{lem-generative}, $\{ \F^j A \}$ generates $\rgr A$ if and only if $\{ \F^j A \}$ generates all integrally supported simple modules.  Since $A$ generates $\X[l]_{l \geq 0}$ and $\Y[l]_{l \leq -1}$ we see that $\F$ is $A$-generative if and only if for all $l \geq 0$, some $\F^j A$ generates $\Y[l]$, and for all $l \leq -1$, some $\F^j A$ generates $\X[l]$.   That is, $\F$ is $A$-generative if and only if 
\begin{equation}\label{cond-I}\begin{split}
\bigcup_{j\neq 0} J_j = \zed.
\end{split}\end{equation}

Now, if $i \not\in I$, then $\Gamma_j \cap (n\zed + i) = \emptyset$, and so if $j>0$, then 
\[J_j \cap (n\zed+i) = \Delta_j \cap (n \zed +i) = \{ i, n+i, \ldots, n(j-1)+i\}.\]
If $j<0$, then $J_j \cap (n\zed+i) = \{ i-n, \ldots, i+nj \}$.  	Thus $(n\zed +i) \subseteq \bigcup_{j \neq 0} J_j$.   

Now suppose that $i \in I$.  Then   if $j >0$, then $ \Gamma_j \cap (n \zed +i) = 	\{ n+i, \ldots, nj+i\}$ and if $j< 0$, then  $ \Gamma_j \cap (n \zed +i)  = 		\{ i, i -n, \ldots, i+n(j+1) \}$.  
We see that $J_j \cap (n \zed + i) = (\Gamma_j \oplus \Delta_j) \cap (n \zed +i ) = \{i, i+nj\}$. Again we have that  $(n\zed +i) \subseteq \bigcup_{j \neq 0} J_j$.   
Thus \eqref{cond-I} is satisfied, and $\F$ is $A$-generative.
\end{proof}

We have gathered the ingredients for one-half of the proof of Theorem~\ref{thm-class}.  The other necessary piece is:
\begin{proposition}\label{prop-SJ}
Let $\F \cong \shift_A^n \inv_J$, with $(J, n)$ an admissible pair.  Then $\End^{\F}_A(A)$ is isomorphic to $S(J,n)$.  
\end{proposition}

Assuming this proposition for the moment, we may now prove Theorem~\ref{thm-class}.

\begin{proof}[Proof of Theorem~\ref{thm-class}]
Proposition~\ref{prop-SJ} tells us that if $(J, n)$ is an admissible pair, then $S(J,n)$ is isomorphic to $\End_A^{\F}(A)$ for $\F = \shift_A^n \inv_J$.  By Lemma~\ref{l-gen}, such an $\F$ is $A$-generative; thus by Theorem~\ref{thm-grMor-pullback}, $S(J, n)$ is graded equivalent to $A$.

Now let $S$ be a graded ring and let  $\Phi: \rgr A \to \rgr S$ be a category equivalence, with quasi-inverse $\Psi$.  Let $\F = \Phi^* \shift_S$, let $m$ be the rank of $\F$, and let $P = \Psi S$.  By Theorem~\ref{thm-grMor-pullback}, $\F$ is $P$-generative.    
By Lemma~\ref{lem-comb2}(1), there is some $\F'$ in the conjugacy class of $\F$ of the form $\F' \cong \shift^n_A \inv_J$ with $(J, n)$ an admissible pair.   By Lemma~\ref{l-gen}, $\F'$ is $A$-generative; thus by Proposition~\ref{prop-Morita-conjugate}, $S = \End^{\F}_R(P)$ is graded Morita equivalent to $\End_A^{\F'}(A)$.  By Proposition~\ref{prop-SJ}, $\End_A^{\F'}(A)$ is isomorphic to $S(J, n)$.  

Now let $(L, r)$ be another admissible pair, and let $\sh{G} = \shift_A^r \inv_L$.   By Proposition~\ref{prop-SJ} and Proposition~\ref{prop-Morita-conjugate}, $S(L, r) \cong \End^{\sh{G}}_A(A)$ and $S(J,n)$ are graded Morita equivalent if and only if $\sh{G}$ and $\F'$ are conjugate in $\Pic(\rgr A)$; but by Lemma~\ref{lem-comb2}(2) this is true if and only if $(L, r)$ and $(J, n)$ have the same necklace type.
\end{proof}

To complete the proof of Theorem~\ref{thm-class}, all that remains is to prove Proposition~\ref{prop-SJ}.  Before doing this,  we give a general lemma allowing us to calculate twisted endomorphism rings explicitly.  
To carry out our computations, we will work in two localizations of $A$.  Let $D$ be the graded quotient ring of $A$ and let $T = A[y^{-1}]$.  If $\sigma$ is the automorphism of $k[z]$ or $k(z)$ that sends $z \mapsto z+1$, we have that $D \cong k(z)[x, x^{-1}; \sigma] = k(z)[y^{-1}, y; \sigma]$ and that $T  \cong k[z][y^{-1}, y; \sigma]$, which we write in this way to emphasize the grading.

\begin{lemma}\label{lem-calcring}
Let $\F$ be an $A$-generative (and therefore even) rank $n$ autoequivalence of $\rgr A$, and let $C$ be the $\F$-twisted endomorphism  ring $\End^{\F}_A(A)$.  Write $\F = \shift^n_A \circ \inv_J$, and define $A$-submodules $M(j) $ of $D$ by:
\[ M(j) = \brackarr{ 
	\Bigl( \prod_{i=1}^{j} \invbrak{J + ni}^{-1} \Bigr) A 	& \mbox{if $j \geq 1$} \\
	A		& \mbox{ if $j=0$} \\
	\Bigl( \prod_{i = 0}^{-j-1} \invbrak{J-ni} \Bigr) A	& \mbox{if $j \leq -1$}	} \]
and a graded vector subspace $C' = \bigoplus_{j \in \zed} C'_j$ of the Veronese ring $D\ver{n}$ by 
\[  C'_j = M(j)_{nj}.\]
Then $C'$ is a subring of $D\ver{n}$ and $C' \cong C$.  
\end{lemma}
\begin{proof}
Since $\inv_J(D) \subseteq D$ is the graded injective hull of $\inv_J A$, then $\inv_J(D) = D$ and so for any $j \in \zed$, we have $\F^jD = \shift^{nj}_A D = D\ang{nj}$; further, if $g: D \to D$ is a graded $A$-module map, then  $\inv_J(g) = g$ and so $\F^j (g) = \shift^{nj}_A(g)$ as maps from $D\ang{nj} \to D\ang{nj}$.

Now, for each $j \in \zed$ there is a natural map $\psi_j: \F^j A \to D\ang{nj}$ given by applying $\F^j$ to the inclusion $\psi_0: A \hookrightarrow D$.  For $j \geq 1$, $\psi_j$ is the composition
\[  \F^j A \to \shift^{nj}_A \invbrak{J - n(j-1)} \circ \cdots \circ  \inv_{J} A \hookrightarrow D\ang{nj},\]
and for $j \leq -1$, $\psi_j$ is the composition
\[  \F^j A\to \shift^{nj}_A \invbrak{J + n(-j)}^{-1} \circ  \cdots \circ \invbrak{J + n}^{-1} A \hookrightarrow D\ang{nj}.\]
Notice that the image of $\psi_j$ is precisely $\shift^{nj}_A M(-j)$.  That is, for any $j \in \zed$ the module $M(j)$ may be identified with $\shift^{nj}_A \F^{-j} A$.  

Let $H$ be the $\zed$-algebra $H = \Hfunct{A}{\bigoplus_{j \in \zed} \F^j A}{\bigoplus_{j \in \zed} \F^j A}$; that is, 
\[H_{ij} = \hom_A(\F^j A, \F^i A).  \]
Suppose that $f: \F^j A \to \F^i A$ is an element of $H_{ij}$.  Then by graded injectivity of $D \ang{ni}$, there is a unique map $\widetilde{f}: D\ang{nj} \to D \ang{ni}$ such that the diagram
\beq\label{starstar}\begin{split} \xymatrix{
\F^j A \ar[r]^{f} \ar[d]_{\psi_j}	& \F^i A \ar[d]^{\psi_i}	\\
D \ang{nj} \ar[r]_{\widetilde{f}}	& D\ang{ni}	}
\end{split} \eeq
commutes.  Let $E = \bbar{D \ver{n}}$ and define a map 
\[\phi: H \to E = \Hfunct{A}{\bigoplus_{j \in \zed} D\ang{nj}}{\bigoplus_{i \in \zed} D\ang{ni}}\]
 by defining $\phi(f) = \widetilde{f}$.   

We claim that $\phi$ is a homomorphism of principal $\zed$-algebras; that is, $\phi$ is a (graded) map of $\zed$-algebras that commutes with the principal automorphisms of $H$ and $E$.  Let $\beta$ be the principal automorphism of $H$ induced from $\F$ and let  $\alpha$ be the canonical principal automorphism of $E$, induced from $\shift^n_A$.   Since $\F$ and $\shift_A$ are automorphisms of $\rgr A$,   in this case the compatibility isomorphisms are trivial, and we have $\beta(f)  = \F(f)$ and $\alpha(\phi(f)) = \shift^n_A(\phi(f))$.

Given  $f \in H_{ij}$  and $g \in H_{jk}$, there is a commutative diagram
\[ \xymatrix{
	\F^k A \ar[r]^{g} \ar[d]_{\psi_k}	& \F^j A \ar[r]^{f} \ar[d]^{\psi_j}	& \F^iA \ar[d]^{\psi_i} \\
	 D\ang{nk} \ar[r]^{\phi(g)} \ar@/_1pc/[rr]_{\phi(f\circ g)}				& D \ang{nj} \ar[r]^{\phi(f)}			& D  \ang{ni}.	} \]
Uniqueness of the lifting $\phi(f \circ g)$ gives us that $\phi(f \circ g) = \phi(f) \circ \phi(g)$, and so $\phi$ is a homomorphism of $\zed$-algebras.

We verify that $\phi$ is a morphism of principal $\zed$-algebras.  Fix $f \in H_{ij}$.  Then $\phi(f)$ is given by the diagram \eqref{starstar}, where $\phi(f) = \tilde{f}$.  
Applying $\F$ to \eqref{starstar}, we obtain:
\beq\label{triplestar} \begin{split}
 \xymatrix{
\F^{j+1} A \ar[rrrrr]^{\F(f) = \beta(f)} \ar[d]_{\F(\psi_j)}	&&&&& \F^{i+1} A \ar[d]^{\F(\psi_i)} \\
D \ang{n(j+1)} \ar[rrrrr]^{\F(\phi(f)) = \shift^{n}_A(\phi(f)) = \alpha(\phi(f))} &&&&& D \ang{n(i+1)}.	}
\end{split}\eeq
But by construction, $\F(\psi_j) = \psi_{j+1}$ and $\F(\psi_i) = \psi_{i+1}$, so \eqref{triplestar} is precisely the diagram giving $\phi(\beta(f))$.  That is, $\phi(\beta(f)) = \alpha(\phi(f))$, and  $\phi$ is a map of principal $\zed$-algebras, as claimed.

In particular, the canonical principal map $\alpha$ on $E$ restricts to a principal map on the image of $\phi$.  We saw that $\alpha$ induces a natural ring structure on $\bigoplus_{j \in \zed} E_{0j}$ that makes it isomorphic to $D\ver{n}$.  By Remark~\ref{rmk-opp}, it likewise  induces a ring structure isomorphic to $D\ver{n}$ on $\bigoplus_{j \in \zed} E_{-j,0}$.  Thus we may restrict the identifications $E_{0,j} = D_{nj} = E_{-j,0}$ to the image of $\phi$ to obtain a subring of $D\ver{n}$.  That is, the natural map 
\[  C_j \cong  \hom_A(A, \F^{-j} A)  \to \hom_A(A, \shift_A^{-nj}(M(j))) = M(j)_{nj} \stackrel{\phi}{\to} E_{-j,0} = D_{nj} \]
gives a ring monomorphism
\[ \phi: C \hookrightarrow D\ver{n}.\]
Therefore, $C' = \im \phi = \bigoplus_{j \in \zed} M(j)_{nj}$  is a subring of $D\ver{n}$ and $C \cong C'$, as claimed.  
\end{proof}

\begin{proof}[Proof of Proposition~\ref{prop-SJ}]
The proof is by direct computation.  
Put $S = \End^{\F}_A(A)$.  By Lemma~\ref{lem-calcring},  we may compute $S$ via:
\[S \cong \bigoplus_{j\in \zed} M(j)_{nj} \subseteq D\ver{n}\]
 where  
 \[
 M(j) = \shift_A^{nj} \F^{-j} A = \brackarr{ 
	\invbrak{(J+n) \oplus \cdots \oplus (J + jn)}^{-1}A	& \mbox{for $j \geq 1$} \\
	A	& \mbox{for $j=0$} \\ 
	 \invbrak{J \oplus (J-n) \oplus \cdots \oplus (J + (j+1)n)}A & \mbox{for $j\leq -1$.}
	} \]

Therefore, if $j>0$ we have
\begin{multline*}
S_j = M(j)_{nj} = (\invbrak{(J+n) \oplus \cdots \oplus (J+jn)}^{-1} A)_{nj} \\
= \Bigl( \prod_{l=1}^j f_{J+nl} \Bigr)^{-1} 
			\cdot (\invbrak{(J+n) \oplus \cdots \oplus (J+jn)} A)_{nj}, 
\end{multline*}
and if $j<0$ we have 
\[S_j =  (\invbrak{J \oplus (J-n) \oplus \cdots \oplus (J +(j+1)n)} A)_{nj}.\]
  If $j =0$ then $S_0 = A_0 = k[z]$.

We compute the terms $S_j$, using Lemma~\ref{lem-form}.  
We first let $j <0$.  
We want the $y^{-nj}$ term of 
\[ \invbrak{J \oplus (J-n) \oplus \cdots 
\oplus (J +(j+1)n)}A = \invbrak{J \cup (J-n) \cup \cdots \cup (J +(j+1)n)} A.\]
Let $L =  (J - n) \cup \cdots \cup (J + (j+1)n)$.  If $i\in L$, then since $0 > i >nj$,  by Lemma~\ref{lem-form}(3) we have $(\inv_i A)_{nj} = y^{-nj} k[z]$.  If $i \in J$, then $(\inv_i A)_{nj} = (z+i) y^{-nj} k[z]$.  That is, if $j<0$ we have by Lemma~\ref{lem-form}(1) that 
\beq\label{j-neg}
S_j  = \bigcap_{i \in J \cup L} \inv_i A 
 = \Bigl( \prod_{i \in J}(z+i) \Bigr) y^{-nj} k[z] \cap y^{-nj} k[z] = f_J y^{-nj} k[z].
\eeq

Now consider the terms for $j>0$.  Let 
\[K = (J+n) \oplus \cdots \oplus (J+nj) = (J+n) \cup \cdots \cup (J+nj).\]
  Put $K' = K \cap \zed_{\leq nj-1}$ and $K'' = K \cap \zed_{\geq nj} = J + nj$.  A similar computation shows that  $(\inv_{K'}A)_{nj} = x^{nj} k[z]$, and $(\inv_{K''}A)_{nj} = f_{K''} x^{nj} k[z] = f_{J+nj} x^{nj} k[z]$.
Now, $K = K' \cup K''$ and we see that 
\[ (\inv_K A)_{nj} = (\inv_{K'} A)_{nj} \cap (\inv_{K''}A)_{nj} = f_{J+nj} x^{nj} k[z].\]
Thus, 
\begin{multline*}	
S_j  = \Bigl( \prod_{l=1}^j f^{-1}_{J+nl}\Bigr) \cdot (\inv_K A)_{nj} 
	= \Bigl( \prod_{l = 1}^j f^{-1}_{J+nl} \Bigr) \cdot f_{J+nj} x^{nj} k[z] \\
	= f_J \cdot \Bigl( \prod_{l=0}^{j-1} f^{-1}_{J+nl}\Bigr) x^{nj} k[z].	\end{multline*}
It is straightforward to verify that this is equal to 
\[ f_J \cdot (f_J^{-1} x^n)^j k[z].\]

Recall the notation that  $\bbar{J} = \{ 0 \ldots n-1\} \smallsetminus J$.  
 Since $x^n = \prod_{i=0}^{n-1} (z+i) y^{-n} = f_J \cdot f_{\bbar{J}} y^{-n}$, we see that for $j > 0$, we have
\beq\label{j-pos}
S_j = f_J \cdot (f_{\bbar{J}} y^{-n})^j k[z].
\eeq
Since $S_0 = k[z]$, combining \eqref{j-neg} and \eqref{j-pos} we have that  $ S \cong \I_W(f_J W)$, 
where 
$ W = k \ang{f_{\bbar{J}} y^{-n}, y^n, z}$.
But it is easy to see that 
$ W \cong \GWA(f_{\bbar{J}}, n)$
--- that is, $S$ is precisely $S(J, n)$.
\end{proof}

\begin{example}\label{eg-ring0}
Let $\F = \shift_A \circ \inv_0$.  By Theorem~\ref{thm-class}, $A$ and $B = \End^{\F}_A(A)$ are graded equivalent; by Proposition~\ref{prop-SJ}, $B$ is isomorphic to 
\begin{equation*}\begin{split}
S(\{0\},1) = \I_T(zT) &= {\brackarr{ 
	 z y^{-n} k[z]		& \mbox{ if $n \geq 1$} \\
	 k[z]							&\mbox{ if $n=0$}\\
	zy^{-n}k[z]		& \mbox{ if $n \leq -1$.}
	} }
\end{split}\end{equation*} 

It is certainly surprising that these two rings are graded equivalent, since their behavior is so different.  It is well-known that $A$ is simple, and is therefore a {\em maximal order}:  there is no ring $S$ with $A \subsetneq S \subseteq Q(A)$ such that $a S b \subseteq A$ for some nonzero $a, b \in A$.  On the other hand, $B$ is neither simple or a maximal order.  Further, $B$ fails the {\em second layer condition} governing relationships among prime ideals and enabling localization (see \cite[Chapter~11]{GW}), while $A$ (trivially) satisfies the second layer condition.    Also, while $A$ has no finite-dimensional modules, $B$ has a 1-dimensional graded representation.

We explore the equivalence 
 $\Upsilon = \Htwist{\F}{A}{A}{\blank}$ between $\rgr A$ and $\rgr B$.    
 We first describe how idealizing affects graded module categories.  Let $R$ be a graded ring with a graded right ideal $I$ that is maximal as a right ideal of $R$; let $S = \I_R(I)$ be the idealizer of $I$ in $R$.  By \cite[Theorem~1.3]{Rob}, the ungraded category $\rmod S$ has one more simple than the category $\rmod R$:  the simple $R$-module $R/I$ becomes a length 2 module over $S$.  Thus in the graded category, idealizing corresponds to replacing the simples $(R/I)\ang{n}$ with pairs of simple modules.  

Now consider the category $\rgr B$.  Since $T$ is strongly graded, by \cite[Theorem~I.3.4]{NV}, the categories $\rgr T$ and $\rmod k[z]$ are equivalent, and the simples in $\rgr T$ are naturally parameterized by the affine line.  The discussion above shows that $\rgr B$ may be represented by an affine line  with double points at every integer:
\begin{equation*}
\xymatrix@R=0pt@C=1.5cm{
& :  \ar[l] \ar@{-}[r] & : \ar@{-}[r]& : \ar@{-}[r] & : \ar[r] & . \\
& -1		& 0			&1			& 2  }
\end{equation*}
 We saw in Lemma~\ref{lem-simple} that this is also a representation of the simple objects in $\rgr A$.

  A quick computation shows that  $\F^nA \cong \inv_{0} \inv_n A$, and since
$\Upsilon (\F^nA) \cong B\ang{n}$, we see that 
$B$ is the unique shift of $B$ that generates $\Upsilon(X)$. Thus $\Upsilon(X)$ is the 1-dimensional module $k$, and  $\Upsilon$ maps the exact sequence
\[ \xymatrix{ 0 \ar[r] & xA \ar[r] & A \ar[r] & X \ar[r] & 0} \]
to 
\[ \xymatrix{ 0 \ar[r] & zT \ar[r] & B \ar[r] & k \ar[r] & 0.} \]
Thus the right ideal $xA$ maps to a two-sided ideal in $B$.

  Since the existence of $\Upsilon$ is rather counterintuitive, it is not surprising that $\Upsilon$ is quite different from the standard examples of  graded Morita equivalences and Zhang twists.   This can also be seen  in the Picard group of $\rgr A$:  the autoequivalence $\shift_A \inv_0 = \Upsilon^* \shift_B$ is not induced from  an automorphism of $\bbar{A}$, by Corollary~\ref{cor-new}.   On the other hand, by Corollary~\ref{cor-automs}, any autoequivalence of $\rgr A$ that corresponds to a Zhang twist or a graded Morita equivalence would be in the image of  \eqref{Aut-Pic}.
\end{example}

The discussion in Example~\ref{eg-ring0} proves Corollary~\ref{icor2}.

\begin{example}\label{eg-idealizerA}
For another example, we note that $A$ is graded equivalent to an idealizer in $A$:  put $n=2$ and $J= \{0\}$.  Then by Proposition~\ref{prop-SJ} the corresponding ring is $C = S(\{0\}, 2) = \I_{W}(zW)$, where $W = W(z+1, 2)$ is isomorphic to $A$ via
\begin{align*}
X & \mapsto x \\
Y & \mapsto 2y \\
z & \mapsto 2z-1
\end{align*}
Thus $C$ is isomorphic to $\I_A((z-1/2)A)$.   

This example is slightly less counterintuitive if we note that passing from $\rgr A$ to $\rgr C$ corresponds to adding a simple module for each half-integer point; thus pictorially the equivalence between $\rgr A$ and $\rgr C$ corresponds to scaling by a factor of 1/2.  It is still surprising that this can be made to work functorially.
\end{example}

\begin{remark}\label{rmk-stack}
Paul Smith has recently shown \cite{Smith-stack} that the category $\rgr A$ is also equivalent to a graded module category over a commutative ring $R$.  The ring $R$ is graded by $\zedfin$, and $\rgr R$ naturally corresponds to the affine line with a stacky $\zed/2\zed$-point at every integer.  (See \cite{Smith-stack} for precise definitions.)  This is plausible from the pictorial representation of $\rgr A$, but it is certainly quite counterintuitive that any module category associated to the Weyl algebra is commutative in any sense!
\end{remark}

We remark that  a similar classification to Theorem~\ref{thm-class} can be carried out for other generalized Weyl algebras, in particular for primitive factors of $U  = U(sl_2(\mathbb{C}))$.  In terms of generators and relations, $U$ is generated over $\mathbb{C}$   by $E$, $F$, and $H$, subject to the relations
\begin{align*}
[H, E] & = 2E 	& [H, F] & = -2F 		& [E,F] & = H.
\end{align*}
Let $\Omega \in U$ be the Casimir element $4EF + H^2 - 2H$.  Then the primitive factors of $U$ are given by factoring out $\Omega + \mu$ for some $\mu \in \mathbb{C}$; we write them as  $U_\lambda = U/(\Omega-\lambda^2 - 2\lambda)$ for $\lambda \in \mathbb{C}$.  If we let $e$ and $f$ be the images of $E$ and $F$, respectively, in $U_\lambda$, and let $h$ be $1/2$ times the image of $H$, then $U_\lambda$ is given in terms of generators and relations by 
\begin{align*}
[h, e] & = e	& fe & = -(h+\frac{\lambda}{2} + 1) (h - \frac{\lambda}{2})	\\
 [h, f] & = -f		& ef & = -(h+ \frac{\lambda}{2})(h - \frac{\lambda}{2} - 1).
\end{align*}
The $U_\lambda$ are clearly (isomorphic to) generalized Weyl algebras.   

Stafford \cite{St} has shown that in this parameterization, $U_\lambda$ is hereditary  when $ \lambda \not \in \zed$; that is, when the roots of $ef$ do not differ by an integer.   If $\lambda \in \zed$ but $\lambda \neq -1$, then $U_\lambda$ has global dimension 2.  If $\lambda = -1$, so the roots of $ef$ coincide, then $U_\lambda$ has infinite global dimension.  (See  \cite{Bav} for a generalization to arbitrary generalized Weyl algebras.)

For all values of $\lambda$, there are involutions of the category $\rgr U_\lambda$ similar to the involutions $\inv_j$ on $\rgr A$.  We will not give the classification of rings graded equivalent to $U_\lambda$, but we will note that if we let $S = U_\lambda[f^{-1}] \cong \mathbb{C}[h][f^{-1}, f; \tau]$, where $\tau(h) = h-1$, then similar calculations to Proposition~\ref{prop-SJ} produce the following result, which we give without proof:
\begin{proposition}\label{prop-U}
Let $\lambda \in \mathbb{C}$, and let $U_\lambda$ and $S$ be the rings defined above.  Then:

$(1)$ If $\lambda \not\in \zed$, then $U_\lambda$ is graded equivalent to $\I_S((h+\lambda/2)(h - \lambda/2 - 1) S)$.  Further, $U_\lambda$ is graded equivalent to $\I_A((z+\lambda + 1)A)$ and to $\I_A((z - \lambda -1)A)$.

$(2)$ If $\lambda=-1$, then $U_\lambda$ is graded equivalent to $\I_S((h-1/2)^2 S)$. 

$(3)$ If $- 1 \neq \lambda \in  \zed$, then $U_\lambda$ is graded equivalent to $\I_S((h+\lambda/2)(h- \lambda/2 -1)^2 S)$.    \qed
\end{proposition}
It is a curiosity that we obtain idealizers in $S$ at a quadratic polynomial in the first two cases and at a cubic polynomial in the third.

\section{The graded $K$-theory of the Weyl algebra}\label{sec-Ktheory}
In this section we complete the analysis of the category $\rgr A$ by computing $K_0(\rgr A)$.  This is not needed to prove Theorem~\ref{ithm1}, but it does give further insight into differences between $\rgr A$ and $\rmod A$.  We also  show relationships between $K$-theory and the decomposition of $\Pic (\rgr A)$ given in  Section~\ref{sec-Pic}.

The ungraded $K$-theory of $A$ is trivial.  Not only is any projective module stably free, but by a result of Webber \cite{W}, if $P$ is any rank 1 projective then $P \oplus A \cong A^2$.  Thus one expects that the graded $K$-theory of $A$ will also be trivial; that is, that $K_0(\rgr A)$ is isomorphic to a countable direct sum $\zed^{(\zed)}$, where the generators $e_j$ correspond to the cyclic projectives $A \ang{j}$.  This is in fact true, but because the isomorphisms $P \oplus A \cong A^2$ are not graded, it does not  follow directly from the ungraded case.  In fact, Webber's result  is false for $\rgr A$:  if $P = x^4 A + (z+1)(z+3) A$, then there are no integers $l, m, n$ such that $P \oplus A \ang{l} \cong A\ang{m} \oplus A \ang{n}$.  However, it is true that all graded projectives are stably free; in this particular example, we have
\beq\label{foo} P \oplus A \ang{3} \oplus A \ang{1} \cong A \ang{4} \oplus A \ang{2} \oplus A.\eeq
Furthermore, we will see that in $\rgr A$, if $Q \oplus Q' \cong Q \oplus Q''$, then $Q' \cong Q''$.

We begin with a preliminary lemma:

\begin{lemma}\label{lem-sumisom}
Let  $J, K, J', K' \in \zedfin$.  Then $\inv_J A \oplus \inv_K A \cong \inv_{J'} A \oplus \inv_{K'} A$ if and only if $J \cap K = J' \cap K'$ and $J \cup K = J'\cup K'$.   In particular, 
\beq\label{K-identity}
 \inv_J A \oplus \inv_K A \cong \invbrak{J \cup K}A \oplus \invbrak{J \cap K} A.
\eeq
\end{lemma}
\begin{proof}
$(\Leftarrow)$.  It is enough to establish \eqref{K-identity}. 
By Proposition~\ref{prop-duality}, $A$ has a unique semisimple factor that is supported on $J \smallsetminus K$.  Call this module $M$; then $M$ is also a factor of $\invbrak{J \cap K}A$ and  $\inv_K A$.  It is easy to see that  there is  a diagram with  exact rows:
\[ \xymatrix{
0 \ar[r] & \invbrak{J \cup K} A \ar[r] & \inv_K A \ar[r] 	& M \ar@{=}[d] \ar[r]	& 0 \\
0 \ar[r] & 	\inv_J A \ar[r]	& \invbrak{J \cap K} A \ar[r]	& M \ar[r] 	& 0. } \]
By Schanuel's lemma we see that we have
$\inv_J A \oplus \inv_K A \cong \invbrak{J \cap K} A \oplus \invbrak{J \cup K} A$, as claimed.

$(\Rightarrow)$.  By \eqref{K-identity} we may without loss of generality assume that $J \subseteq K$ and $J' \subseteq K'$.  Now comparing simply supported factors of length 2, which must be isomorphic, we see immediately that $J = J'$ and $K = K'$.
\end{proof}

\begin{proposition}\label{prop-K0}
Let $K_0= K_0(\rgr A)$.   Then $K_0 \cong \zed^{(\zed)}$ is  a direct sum of countably many copies of $\zed$, and is generated by the equivalence classes of the shifts $\{ A\ang{n}\}_{n \in \zed}$.  Further, if $P$ and $P'$ are finitely generated graded projective modules, then $[P] = [P']$ in $K_0$ if and only if $P \cong P'$.
\end{proposition}
\begin{proof}
We will prove the second statement first.  Suppose that $P \oplus Q \cong P' \oplus Q$, where $P$, $P'$, and $Q$ are finitely generated graded projective modules.   This implies that $P$ and $P'$ have the same rank, say $m$.  Further, if $S$ is an integrally supported simple module, then 
\begin{multline*}
\dim_k \hom_A(P, S) = \dim_k \hom_A(P \oplus Q, S) - \dim_k \hom_A(Q, S) \\
= \dim_k \hom_A(P', S). 
\end{multline*}
Using Lemma~\ref{lem-sumisom}, we may write 
\[P \cong \inv_{J_1} A \oplus \cdots \oplus \inv_{J_m} A,\]
where $J_1 \subseteq \cdots \subseteq J_m$.  Likewise we have
\[P' \cong \inv_{K_1} A \oplus \cdots \oplus \inv_{K_m} A,\]
where $K_1 \subseteq \cdots \subseteq K_m$.
 The sets $J_i$ may be read off from the dimensions of the vector spaces $\hom_A(P, S)$ as $S$ varies over all integrally supported simples.  Thus $K_i = J_i$ for all $i$, and $P \cong P'$.

In particular, the classes $[A\ang{n}]$ in $K_0$ are all distinct.  Let $G$ be the subgroup of $K_0$ that they generate; we will show that $G = K_0$.  It suffices to show that $G$ contains all  rank 1 graded projective modules; let $P$ be such a module.  By Lemma~\ref{lem-asymptotic}, $F_j(P) = Y\ang{j}$ for all $j \ll 0$, and $F_j(P) = X\ang{j}$ for all $j \gg 0$.  Let 
\[n = \min \{ j \in \zed \st  F_j(P) = X\ang{j} \}.\]
  It clearly suffices to prove that $[P \ang{-n}]$ is in $G$; that is, without loss of generality we may assume that $P \cong \inv_J A$, with $J \subseteq\zed_{\geq 1}$.

We induct on $\#J$.  If $J = \emptyset$, then $[P] = [A]$ is in $G$.  So assume that $\#J \geq 1$, and that for all $I$ with $\#I < \#J$ and $I \subseteq \zed_{\geq 1}$, we have $[\inv_I A] \in G$.  Let $m = \max J$.  Then  by Lemma~\ref{lem-sumisom}, we have that
\begin{multline*} P \oplus A\ang{m} \cong \invbrak{J \cap \{ 0, \ldots, m-1\}}A \oplus \inv_{\{0, \ldots, m\}} A \\
	\cong \invbrak{J \cap \{0, \ldots, m-1\}}A \oplus A \ang{m+1}. \end{multline*}
By induction, $[\invbrak{J \cap \{ 0, \cdots, m-1\}}A] \in G$, and so $[P] \in G$.
\end{proof}

We define the  {\em graded reduced Grothendieck group} $\widetilde{K_0} = \widetilde{K_0} (\rgr A)$  to be $K_0/([A])$.  Then we have:

\begin{corollary}\label{cor-K_0-red}
There is a unique group homomorphism $\theta: \widetilde{K_0} \to \Picz(\rgr A)$ such that $\theta([\inv_JA]) = \inv_J$ for all $J$.  Further, $\theta$ 
induces an isomorphism
\[ \theta \otimes 1: \widetilde{K_0} \otimes \zed/2\zed \to \Picz(\rgr A).\]
\end{corollary}
\begin{proof}
It is enough to prove that $\theta \otimes 1$ is well-defined and is an isomorphism.  
By Proposition~\ref{prop-K0} and Lemma~\ref{lem-sumisom}, the relations on $\widetilde{K_0}$ are generated by relations of the form
\[
[\inv_J A] + [\inv_K A] = [\invbrak{J \cap K}A]+ [\invbrak{J \cup K}A] = 2[\invbrak{J\cap K} A] + [\invbrak{J \oplus K} A] 
\]
and if $J \cap K = \emptyset$, then 
\[ [\inv_J A] + [\inv_K A] = [\invbrak{J \cup K} A].\]
Thus the relations on $\widetilde{K_0} \otimes \zed/2\zed$ are of the form:
\[ 
[\inv_J A] \otimes 1 + [\inv_K A] \otimes 1 =  [\invbrak{J \oplus K} A] \otimes 1
\]
and clearly we have an isomorphism to $\Picz(\rgr A)$ as claimed.
\end{proof}

The relations given above show that  $\widetilde{K_0}$ is isomorphic to the group generated by the semigroup of finite {\em multisets} of integers under the operation of union; for example, \eqref{foo} may be written
\[ \inv_{\{1,3\}}A \oplus \inv_{\{0,1,2\}} A \oplus \inv_{0} A \cong \inv_{\{0,1,2,3\}}A \oplus \inv_{\{0,1\}}A \oplus \inv_\emptyset A.\]
Note  that the multisets $\{ 1, 3, 0, 1, 2, 0 \}$ and $\{ 0, 1, 2, 3, 0, 1\}$ are equal.

\end{document}